\documentclass[11pt,oneside,a4paper]{amsart}

\usepackage[
	a4paper,
	margin=1in
]{geometry}
\usepackage{amsmath,amssymb,amsthm}
\usepackage{float}
\usepackage{wrapfig}
\usepackage{needspace}
\usepackage{enumitem}
\usepackage{placeins}
\usepackage{graphicx}
\usepackage{xcolor}
\usepackage{aliascnt}
\usepackage{tikz}
\usetikzlibrary{external,intersections}
\tikzset{external/up to date check=md5}
\definecolor{linkturquoise}{HTML}{006D70}
\definecolor{citepink}{HTML}{9B1B5A}
\usepackage[
colorlinks=true,
linkcolor=linkturquoise,
citecolor=citepink,
urlcolor=linkturquoise
]{hyperref}
\usepackage[nameinlink,noabbrev]{cleveref}

\usepackage{comment}
\newtheorem{theorem}{Theorem}[section]
\newtheorem*{maintheorem}{Theorem A}
\newaliascnt{lemma}{theorem}
\newtheorem{lemma}[lemma]{Lemma}
\aliascntresetthe{lemma}
\newaliascnt{proposition}{theorem}
\newtheorem{proposition}[proposition]{Proposition}
\aliascntresetthe{proposition}
\newaliascnt{corollary}{theorem}
\newtheorem{corollary}[corollary]{Corollary}
\aliascntresetthe{corollary}
\theoremstyle{definition}
\newaliascnt{definition}{theorem}
\newtheorem{definition}[definition]{Definition}
\aliascntresetthe{definition}
\newaliascnt{example}{theorem}
\newtheorem{example}[example]{Example}
\aliascntresetthe{example}
\newtheoremstyle{boldremark}
{\topsep}{\topsep}
{\normalfont}{}
{\bfseries}{.}
{.5em}{}
\theoremstyle{boldremark}
\newaliascnt{remark}{theorem}
\newtheorem{remark}[remark]{Remark}
\aliascntresetthe{remark}
\newtheorem*{roadmap}{Roadmap}

\makeatletter
\let\c@equation\c@theorem

\let\c@figure\c@theorem

\makeatother
\AtBeginDocument{%
	\makeatletter
	\makeatother
}

\crefformat{equation}{(#2#1#3)}
\Crefformat{equation}{(#2#1#3)}
\crefrangeformat{equation}{(#3#1#4)--(#5#2#6)}
\Crefrangeformat{equation}{(#3#1#4)--(#5#2#6)}
\crefmultiformat{equation}{(#2#1#3)}{ and (#2#1#3)}
{, (#2#1#3)}{, and (#2#1#3)}
\Crefmultiformat{equation}{(#2#1#3)}{ and (#2#1#3)}
{, (#2#1#3)}{, and (#2#1#3)}

\crefname{theorem}{theorem}{theorems}
\Crefname{theorem}{Theorem}{Theorems}
\crefname{lemma}{lemma}{lemmas}
\Crefname{lemma}{Lemma}{Lemmas}
\crefname{proposition}{proposition}{propositions}
\Crefname{proposition}{Proposition}{Propositions}
\crefname{corollary}{corollary}{corollaries}
\Crefname{corollary}{Corollary}{Corollaries}
\crefname{definition}{definition}{definitions}
\Crefname{definition}{Definition}{Definitions}
\crefname{example}{example}{examples}
\Crefname{example}{Example}{Examples}
\crefname{remark}{remark}{remarks}
\Crefname{remark}{Remark}{Remarks}

\newcommand{\C}{\mathbb C}
\newcommand{\R}{\mathbb R}
\newcommand{\Z}{\mathbb Z}
\newcommand{\ord}{\operatorname{ord}}
\newcommand{\rebr}{\operatorname{rebr}}
\newcommand{\imbr}{\operatorname{imbr}}
\newcommand{\mt}{\operatorname{mt}}
\newcommand{\dbl}{\operatorname{dbl}}

\newcommand{\bgamma}{\boldsymbol\gamma}
\newcommand{\alphacp}[1]{\alpha_{\bgamma,#1}^{\mathrm{cp}}}
\newcommand{\alpharb}[1]{\alpha_{\bgamma,#1}^{\mathrm{rb}}}

\title{Real morsifications via the trace map}
\author{Pablo Portilla Cuadrado}
 \thanks{The author is supported by
	RYC2022-035158-I, funded by MCIN/AEI/10.13039/501100011033 and by the
	FSE+. He also thanks the support from the grant PID2024-156181NB-C32 funded
	by  MICIU/AEI/10.13039/501100011033 and ERDF/EU}

\begin{document}

	\begin{abstract}
		Every reduced real plane curve singularity admits a real morsification.
		This settles a question of A'Campo and Gusein-Zade, later stated as
		conjectures by Leviant--Shustin and by
		Fomin--Pylyavskyy--Shustin--Thurston. In particular we  overcome the
		main obstruction that remained
		posed by conjugate pairs of nonreal branches.  Our new main ingredient
		is a construction that produces the divide from a nodal smoothing of
		two normalization disks. This is what we call the {\em trace map}.  For
		real
		branches, it
		recovers Gusein-Zade's construction using Chebyshev polynomials.  For
		pairs of complex conjugate branches with distinct tangents, the construction
		gives an explicit formula for the divide in terms of the Puiseux data.
		The general
		method consists in a delicate combination of the trace map with
		A'Campo's translations and contractions
		to produce divides and real morsifications for all reduced real plane
		curve
		singularities.
	\end{abstract}

	\maketitle

\section{Introduction}
\label{sec:introduction}

\subsection*{The real morsification problem}

Let $(C,0)$ be a reduced plane curve germ.  We write $\delta(C,0)$ for its
delta invariant.  Over $\C$, in a maximal nodal deformation of $(C, 0)$,
every sufficiently close fiber other than the central one has $\delta(C,0)$
nodes and no other singularities.  For a deformation over $\R$, a node is
either real
hyperbolic, with local model $xy=0$, real elliptic, with local model
$x^2+y^2=0$, or nonreal and paired with its complex conjugate.  Only the
real nodes occur in the real zero set: a hyperbolic node is a transverse
crossing, while an elliptic node is an isolated point.  A real nodal
deformation is a real morsification if, for every sufficiently small positive
parameter, all critical points of the corresponding holomorphic function are
real, those of its real restriction are Morse, and those on the zero level are
exactly the saddles.

Write $\imbr(C,0)$ for the number of conjugate pairs of nonreal branches of
$(C,0)$.  Leviant and Shustin proved that a real nodal deformation has at
most $\delta(C,0)-\imbr(C,0)$ real hyperbolic nodes.  They also proved that
it is a real morsification precisely when it attains this bound and has no
other singularities
\cite[Lemma 2 and Remark 4]{LeviantShustin2018Morsifications}.  In particular, 
for every positive parameter sufficiently close to the origin, the closure of 
the real locus of the
corresponding nearby fiber is therefore a divide ({\em partage} in the sense of
\cite{ACampo1975MonodromyDeploiement}).

\begin{maintheorem}
	Every reduced real plane curve germ  $(C,0)$  admits a
	real morsification.  Equivalently, there is a real nodal deformation
	$(C_s)$ of a representative of $(C,0)$ in a Milnor ball such that, for
	every $s>0$ small enough, the curve $C_s$ has exactly
	\[
		\delta(C,0)-\imbr(C,0)
		\]
		real hyperbolic nodes and no other singularities.  In particular, the
		closure of its real locus in the closed real Milnor disk is a divide.
\end{maintheorem}

It is important to keep the prescribed real form. Indeed, every topological 
type of a complex
plane curve has a representative all of whose branches are real. Even if the 
problem of finding real morsifications was already solved for plane curves all 
of whose branches are real, this does not solve the problem above.  The 
original germ may have branches exchanged by conjugation, and
two real forms of the same complex singularity may have different real zero
sets and different distributions of real and nonreal nodes.

\subsection*{Historical account}

One can find a precursor to this problem in Charlotte Angas Scott's work of
1892 and 1893 on higher plane curve singularities.  She described a
{\em penultimate} form of a singularity involving nodes and {\em evanescent
loops}
\cite[pp.~301--302]{Scott1892HigherSingularities}
\cite[p.~221]{Scott1893NatureEffectSingularities}.  Existence of real
morsifications was proven first for totally real germs, that is, curve germs 
whose complex branches
are all real.  For these germs the bound above is  exactly $\delta(C,0)$.  
Gusein-Zade
constructed real morsifications for branches, by perturbing real 
parametrizations of the
branches using Chebyshev polynomials
\cite[Sections 4--5, especially Theorem 4]
{GuseinZade1974DynkinDiagrams}.  A'Campo independently gave a geometric 
construction
based on an embedded resolution, by an inductive sequence of translations of 
the strict transforms,
and contracions of the exceptional components
\cite[Section 2]{ACampo1975MonodromyDeploiement}.  For a recent account of
both  constructions, see
\cite[Sections 5.2--5.3]{ACampoPortilla2025Divides}.

The problem of existence of real morsifications remained for prescribed real 
forms with nonreal
branches, equivalently with pairs of complex conjugate branches.  Leviant and 
Shustin proved existence when the following condition
holds.  At every real point of the minimal resolution tree and for each
nonreal tangent line, the branches of the corresponding strict transform
tangent to that line form the union of a Newton nondegenerate germ and a germ
with smooth branches
\cite[Theorem 1]{LeviantShustin2018Morsifications}.  The {\em simplest} 
singularity
not covered by their theory consists of two transverse complex conjugate 
branches with characteristic Puiseux pairs $(2,3), (2, 7)$
\cite[Example 1(3)]{LeviantShustin2018Morsifications}.  In
\Cref{ex:leviant-shustin-first-missing-example}, we deal with this singularity 
and
compute a divide for it.  

The divide of a real morsification encodes topological data of the complex
singularity.  Gusein-Zade computed the intersection matrix of a distinguished
basis of vanishing cycles from the separatrices and families of trajectories of
the real gradient flow
\cite[Theorem 1]{GuseinZade1974IntersectionMatrices}.  A'Campo described a link
$L(D)\subset S^3$ for each divide $D$.  For connected $D$, he proved
that $L(D)$ is fibered.  He also computed its monodromy as an ordered product of
right handed Dehn twists around curves obtained from the double points and 
interior
regions of $D$
\cite[Theorems 1 and 2 and Section 4]{ACampo1998GenericImmersions}.
  When
$D$ is the divide of a real morsification of a totally real plane curve
singularity, $L(D)$ is isotopic to the singularity link and its fiber surface is
diffeomorphic to the Milnor fiber
\cite[Theorems 1 and 2]{ACampo1999RealDeformations}.  The Dehn twists around
these curves generate the geometric monodromy group of the unfolding of the 
singularity, and their
ordered product is the local geometric monodromy
\cite[Theorem 3]{ACampo1999RealDeformations}.  Actually,  the
curves in the factorization form a distinguished system of quadratic vanishing
cycles
\cite[Section 3]{ACampo2001QuadraticVanishingCycles}. Couture and Perron 
described an algorithm that produces a braid
whose closure is the algebraic link $L(D)$ 
\cite{CouturePerron2000RepresentativeBraids}.

For a morsification of a
totally real singularity, Balke and Kaenders proved that the abstract
A'Campo--Gusein-Zade diagram determines the complex topological type, provided
that the divide is a partition in their sense
\cite[Theorem 2.5 and Corollary 2.6]{BalkeKaenders1996CoxeterDynkin}. Leviant
and Shustin subsequently removed the partition assumption in the
theorem of Balke and Kaenders and the restriction to totally real
singularities.  They proved that the diagram of an arbitrary real morsification
determines the weak real topological type: the complex topological type, the
partition of the branches into real branches and conjugate pairs, and the cyclic
order of the real branches
\cite[Section 4, especially Theorem 3]{LeviantShustin2018Morsifications}.

We mention some other applications and generalizations of divides.  Ishikawa
proved that the fiber surface of a connected divide is obtained from a disk by
successively plumbing positive Hopf bands \cite{Ishikawa2002Plumbing}, and
Hirasawa gave a procedure for drawing the associated link together with a
fiber surface
\cite{Hirasawa2002Visualization}.  Gibson and Ishikawa introduced oriented
divides, which realize every oriented link in $S^3$, and free divides
\cite[Definition 1.1 and Theorem 1.2]
{GibsonIshikawa2002OrientedDivides}
\cite{GibsonIshikawa2002GordianIntervals}.  Kawamura defined graph divides and
proved that their links are quasipositive
\cite[Definition 3.1 and Theorem 5.1]{Kawamura2004GraphDivides}.  Sugawara and
Yoshinaga defined divides with cusps and used their links to describe Kirby
diagrams for the complements of complexified real affine line arrangements
\cite[Definition 3.1 and Theorem 5.1]
{SugawaraYoshinaga2022DividesCuspsKirby}. Later, Sugawara showed that a link 
comes from a divide with cusps if 
and only if a 
nontrivial
involution of $S^3$ preserves the link setwise, preserves its orientation, and
has fixed points
\cite[Theorem 1.1]{Sugawara2025DividesCusps}.

Fomin--Pylyavskyy--Shustin--Thurston related divides to quivers and other 
combinatorial objects like plabic
graphs.  For a divide
$D$ coming from a real plane curve, they obtain a quiver $Q(D)$ by
orienting the A'Campo--Gusein-Zade diagram and forgetting its vertex colors
\cite[Definitions 2.1, 3.1, and 4.2]
{FominPylyavskyyShustinThurston2022MorsificationsMutations}.  For an algebraic
divide $D$, its quiver $Q(D)$ determines the complex topological type
\cite[Theorem 4.4]
{FominPylyavskyyShustinThurston2022MorsificationsMutations}.  Their main
conjecture says that the quivers associated with two real morsifications are
mutation equivalent if and only if the corresponding singularities have the
same complex topological type
\cite[Conjecture 5.5]
{FominPylyavskyyShustinThurston2022MorsificationsMutations}.  In particular,
this conjecture predicts that the quivers associated with real morsifications
of different real forms of the same complex singularity are mutation
equivalent.  Theorem~A guarantees that every such real form admits a real
morsification.

\subsection*{The trace map and a sketch of the proof}

Let us first describe the trace map for a conjugate pair $Q,\overline Q$ of
nonreal branches with distinct tangent lines.  Let $\gamma_Q$ and
$\gamma_{\overline Q}$ be the primitive parametrizations chosen in
\Cref{sec:conjugate-pair-case}.  On the smoothing
$\mathcal A=\{(u,v,s):uv=s\}$ of the nodal source $uv=0$, define the
{\em trace map} by
\[
	\begin{aligned}
	\widehat\Phi:\mathcal A&\longrightarrow\C^2,\\
	(u,v,s)&\longmapsto\gamma_Q(u)+\gamma_{\overline Q}(v).
	\end{aligned}
\]
The map $\widehat\Phi$ is equivariant for the involution on the source
$(u,v,s)\mapsto(\overline v,\overline u,\overline s)$ and the ambient real
structure, as explained in \Cref{sec:trace-map}.  For $r>0$, the points of
$\mathcal A_{r^2}=\{(u,v):uv=r^2\}$ fixed by the involution on the source form
the circle
$(u, v) = (re^{i\theta}, re^{-i\theta})$.  By equivariance, the image of this
fixed circle under $\widehat\Phi$ lies in the real plane. Moreover, a
parametrization
of this image is
\[
	\alpha_r(\theta)
	=
	\gamma_Q(re^{i\theta})+
	\gamma_{\overline Q}(re^{-i\theta}).
\]
By \Cref{thm:distinct-tangent-conjugate-pair-predivide}, for every sufficiently
small $r>0$, the image $\alpha_r(\R/2\pi\Z)$ is a predivide for
$Q\cup\overline Q$.  We now choose the adapted coordinates $(p,q)$ of
\Cref{eq:distinct-conjugate-adapted-p-coordinate,eq:distinct-conjugate-real-structure}.
If $m$ is the multiplicity of $Q$ and
$\gamma_Q(u) = (u^m, \varphi(u))$, where
$\varphi(u)=\sum_{n>m}a_nu^n$ is convergent, set
\[
	z_r(\theta)
	=e^{im\theta}
	+\sum_{n>m}\overline{a_n}r^{n-m}e^{-in\theta}.
\]
Then $\alpha_r=L_r\circ z_r$, where
$L_r(z)=(r^mz,r^m\overline z)$ is a real linear isomorphism onto the real
plane.  Thus, to draw this predivide, it is enough to plot the image
$z_r(\R/2\pi\Z) \subset \C$ and then apply $L_r$ (see
\Cref{eq:conjugate-zr-definition,eq:conjugate-alpha-zr-reduction,eq:conjugate-Lr-definition}).
The map $\alpha_r$ is an immersion, and its multiple points have pairwise
distinct tangent lines.  Exactly $\delta(Q\cup\overline Q,0)-1$ unordered
pairs of distinct points of the source map to the same point.  In the case of
real branches, taking a quotient of this trace map recovers Gusein-Zade's Chebyshev
construction, as shown in \Cref{sec:chebyshev-normalization}.

We give the global construction in
\Cref{roadmap:global-construction}.  We blow up the required real infinitely
near points and choose real arcs in the normalizations of the real strict
transforms.  We
then contract the exceptional components in reverse order.  During the
descent, we consider each point $c$ that is the last common real infinitely
near point of some conjugate pair.  The pairs associated with $c$ form the
packet $\mathcal P_c$. We partition it into the unordered pairs of conjugate 
tangent lines.
At $c$, we construct the circles obtained from the trace map and study their
intersections with
the real arcs and the real exceptional components through $c$.

We choose the radii of all circles in a packet as functions of the same
positive parameter.  With this choice, the parametrizations have the required
numbers of pairs of distinct points of the source with the same image.  The
corresponding self-intersections, the intersections between distinct circles,
and the intersections with the real arcs are transverse; see
\Cref{thm:common-tangent-annuli,thm:distinct-tangent-pair-families}.  After
fixing these parametrizations, we replace the parameter by a large enough power
of the global parameter $\lambda$.  Translations with larger powers preserve
the intersections already obtained and make the intersections with the
remaining exceptional divisor distinct and transverse.  After the final
contraction, we can further perturb to split any ordinary point of multiplicity 
$k$ into
$\binom{k}{2}$ nodes. This construction produces at least
$\delta(C,0)-\imbr(C,0)$ nodes.  A comparison with the Euler characteristic
of the Milnor fiber gives the reverse inequality for the sum of the local delta
invariants of the deformed curve.   Thus,
for every sufficiently small nonzero parameter, these nodes are all the
singular points of the corresponding fiber.  For every 
positive (and small) real parameter, they are real hyperbolic.  So the
constructed real 
deformation is a real morsification of the original singularity.

\subsection*{Organization of the paper}

Let us finish by indicating where the different parts of the proof appear.
The conventions for real curve germs, deformations, predivides, and resolution
are collected in \Cref{sec:conventions-definitions}.  We introduce the trace
map and study its real fibers in \Cref{sec:trace-map}.  The case of a real
branch is considered in \Cref{sec:chebyshev-normalization}.  Conjugate pairs with distinct
tangent lines are considered in \Cref{sec:conjugate-pair-case}. This section
also contains the first example left open by Leviant and Shustin.  The assembly
of the orbits of branches along the resolution is carried out in
\Cref{sec:global-predivide}, where we prove Theorem~A.  Finally,
\Cref{sec:worked-seven-branch-example} is devoted to a germ with seven branches
and gives an explicit polynomial that defines a real morsification.

\subsection*{AI disclosure} The main mathematical ideas of this manuscript were 
conceived and developed by the author. Including the original idea of the trace 
map and the strategies for the proofs of all the theorems. ChatGPT 5.5 was used 
to formalize the many estimates, and also for the computations of the examples 
and their figures. It was also used for general editing and formatting formulas 
on the final text. The author remains responsible for all the content in the 
article.

\subsection*{Aknowledgements} I wish to thank Norbert A'Campo for many 
enriching conversations and for helping me make the historical account above 
more 
complete.
\section{Conventions and notation}
\label{sec:conventions-definitions}

The ambient germ is $(X,0)=(\C^2,0)$. We use the notation $(X,0)$ to
emphasize that the coordinates used will not be always the standard ones from
$(\C^2,0)$: we will deal with linear change of coordinates.

\begin{definition}\label{def:real-structure}
		A \textbf{real structure} on a complex analytic space $Y$ is an antiholomorphic
		automorphism $\sigma_Y:Y\to Y$ such that
	$
		\sigma_Y^2=\operatorname{id}_Y.
	$
	That is, an antiholomorphic involution.
	Its fixed locus is the \textbf{real part} of $Y$, denoted by
	\[
		Y_\R=\operatorname{Fix}(\sigma_Y).
		\]
\end{definition}

\begin{remark}\label{rem:local-coordinates-real-structure}
	We define real structures intrinsically because our constructions use 
	linear changes of holomorphic coordinates, so there is no
	preferred
	coordinate system and this language seems more natural. Nonetheless, here 
	we explain that conjugation is, in some sense, the only real structure in 
	an open ball.  If $Y$ is smooth 
	of complex dimension $n$ at
	$c\in Y_\R$, then the points fixed by the real analytic involution
	$\sigma_Y$ form a smooth real analytic manifold near $c$, with tangent space
	\[
		T_cY_\R=\operatorname{Fix}(d\sigma_{Y,c}).
	\]
	 We choose real analytic coordinates on
	$Y_\R$ near $c$ and extend them holomorphically.  The uniqueness of this
	extension gives
	$z\circ\sigma_Y=\overline z$. That is, up to change of coordinates, real 
	structure can always be thought of as conjugation.
\end{remark}

\begin{definition}
		A complex plane curve germ $(C,0)\subset(X,0)$, in a space with ambient
		real
	structure $\sigma_X$, is a \textbf{real plane curve germ} if
	$\sigma_X(C) = C$.  Equivalently, in
	coordinates for which $\sigma_X$ is ordinary conjugation, $C$ is defined
	by a real germ $F\in\R\{x,y\}$.
	We denote the \textbf{real locus} of $C$ by
	$C_\R=\operatorname{Fix}(\sigma_X|_C)=C\cap X_\R$.
\end{definition}

A \textbf{branch} of $(C, 0)$ is the germ of an irreducible component.  If
$(C,0)$ is real, the ambient real structure permutes its branches.  A branch
that is fixed by this action is a \textbf{real branch}. All the other branches 
are 
grouped in
\textbf{conjugate pairs} $Q,\overline Q$.  A \textbf{primitive
parametrization} of a branch is a normalization map from a small disk.
Equivalently, it is injective near the origin.  We call its source the
\textbf{normalization disk}.  The \textbf{multiplicity} $\mt(C,0)$ is the
order of the lowest nonzero homogeneous part of a local equation of $C$.
Equivalently, it is the intersection multiplicity with a line through the
origin that is not contained in the tangent cone of $C$.

For a real branch $B$, we say that the  parametrization $\gamma$ is written in
\textbf{real Newton--Puiseux form} if real analytic coordinates $(x,y)$ in the target and a
real coordinate $w$ on the normalization disk have been chosen so that
\[
	\gamma(w)
	=
	\left(w^m,\sum_{j>m}a_jw^j\right),
	\qquad a_j\in\R,\qquad m=\mt(B,0).
\]
The coordinates in the target and the parameter on the normalization are part
of this input.
We follow Wall for intersection multiplicities
\cite[Section 1.2]{Wall2004SingularPoints} and for parametrizations,
normalization, multiplicity, and tangents
\cite[Theorem 2.2.6 and Lemmas 2.3.1--2.3.3]{Wall2004SingularPoints}.

\begin{definition}\label{def:real-tangent}
	A complex plane branch has a \textbf{real tangent} if its tangent line is
	invariant under the ambient real structure.  Otherwise its tangent is not
	real.  In coordinates for which the real structure is ordinary conjugation,
	a real tangent line is defined by a real linear equation.
	We use the same symbol $\sigma_X$ for the action of the ambient real
	structure on tangent lines and write
	\[
	\overline T:=\sigma_X(T)
	\]
	for the conjugate of a tangent line $T$.
\end{definition}

If $A$ and $B$ are distinct branches defined by  $f_A$ and
$f_B$, their \textbf{local intersection multiplicity} is defined by
\[
(A\cdot B)_0
=
\dim_\C\C\{x,y\}/(f_A,f_B).
\]

\begin{definition}
	Let $(C, 0)$ be a reduced complex plane curve germ defined by $F$, let
	$\mathcal O_C=\C\{x,y\}/(F)$
	be its local ring, and let $\widetilde{\mathcal O}_C$ be its
	normalization.  The \textbf{$\delta$ invariant} is the nonnegative integer
	\[
	\delta(C,0)
	=
	\dim_\C\bigl(\widetilde{\mathcal O}_C/\mathcal O_C\bigr).
	\]
\end{definition}

Suppose a curve germ has branches $B_1,\ldots,B_r$. Then, by applying the 
formula
$\delta(D\cup D') = \delta(D)+\delta(D')+(D\cdot D')_0$
\cite[\S~6.6, p.~152]{Wall2004SingularPoints} iteratively, we get the
following formula for $\delta(C,0)$ in terms of its branches:
\begin{equation}\label{eq:delta-branch-formula}
\delta(C,0)
=
\sum_i\delta(B_i,0)
+
\sum_{i<j}(B_i\cdot B_j)_0.
\end{equation}

\begin{definition}
  We denote by $\rebr(C,0)$  the
	number of real branches of $(C,0)$ at the origin. And we denote by
	$\imbr(C,0)$ the number of complex
	conjugate pairs of branches.
\end{definition}

Let $(C, 0) \subset (X, 0) = (\C^2, 0)$ be a plane curve singularity. Let 
$F:\C^2\to\C$ define a representative of $(C,0)$ near $0$.  A \textbf{Milnor 
ball} for
$(C,0)$ is an open Euclidean ball
$\mathbb B=\mathbb B_\epsilon$, centered at $0$ such that the punctured curve
curve
$\{F = 0\}\cap(\mathbb B\setminus\{0\})$ is smooth and such that the closure of
$\{F=0\}\cap\mathbb B$ meets each sphere
$\partial\mathbb B_{\epsilon'}$ transversely for all $0<\epsilon'\le\epsilon$. 
An 
$\epsilon$ with this property is called a  \textbf{Milnor radius}.

\begin{definition}
	Let $\mathbb B$ be a Milnor ball for $(C,0)$.  The corresponding
	\textbf{representative of the germ in the Milnor ball} is
	\[
		\{F=0\}\cap\mathbb B.
	\]
	Its closure in $\overline{\mathbb B}$ meets
	$\partial\mathbb B$ transversely, and its only possible singular point is
	the origin.  If the
	germ is real, we choose $\mathbb B$ invariant under $\sigma_X$ and write
	$\mathbb B_\R := \mathbb B\cap X_\R$.  Its closure in $X_\R$ is
	$\overline{\mathbb B}_\R := \overline{\mathbb B}\cap X_\R$.
	We call $\overline{\mathbb B}_\R$ the \textbf{closed real Milnor disk}.  In
	coordinates in which $\sigma_X$ is conjugation, these sets are
	$\mathbb B\cap\R^2$ and $\overline{\mathbb B}\cap\R^2$, respectively.
\end{definition}

\begin{definition}
	Choose a reduced real equation $F\in\R\{x,y\}$ for $(C,0)$ and a
	 Milnor ball $\mathbb B$ invariant under the real
	structure.  A \textbf{real deformation} of the representative $C$ defined
	by $F$ is a
	convergent power series
	\[
	\mathcal F(x,y,s)\in\R\{x,y,s\}
	\]
	such that $F_0=F$, where $F_s(x,y):=\mathcal F(x,y,s)$.  Here $s$ is
	a coordinate on a small disk $\Delta_s$, and conjugation in $(x,y,s)$
	is the real structure on the family.  We denote by
	\[
	C_s=\{F_s=0\}\cap \mathbb B
	\]
	the fibers of the deformation.
	For real $s$, we denote by $C_{s,\R}=C_s\cap\mathbb B_\R$ the real part of
	these fibers.
\end{definition}

\begin{definition}
	Let $P$ be a reduced complex plane curve and let $p\in P$.  We call
	$p$ an \textbf{ordinary double point} if there are
	complex analytic coordinates $(u,v)$ centered at $p$ in which a reduced
	local equation of $P$ is $uv=0$.
\end{definition}

\begin{definition}
	Let $P$ be a representative of a reduced real plane curve and let
	$p\in P_\R$ be an ordinary double point.  We call $p$ a
	\textbf{real hyperbolic node} if there are real analytic coordinates $(X,Y)$
	centered at $p$ in which a real local equation of $P$ is
	\[
	XY.
	\]
	In particular, $P_\R$ has two transverse real arcs through $p$.
\end{definition}

\begin{definition}
	Let $\mathcal F$ be a real deformation of the fixed representative of
	$(C,0)$ in the Milnor ball.  The deformation is \textbf{real nodal} if,
	after possibly shrinking $\Delta_s$, the following hold:
	\begin{enumerate}
		\item for each $s\in\Delta_s$, the zero set $C_s=\{F_s=0\}$ is  
		transverse 
		to $\partial\mathbb B$ (and smooth along it),
		\item for every $s\in\Delta_s\setminus\{0\}$, the set
			$\operatorname{Sing}(C_s)$ consists only of nodes, and
		\item the number $\#\operatorname{Sing}(C_s)$ is independent of
			$s\in\Delta_s\setminus\{0\}$.
	\end{enumerate}
	When speaking of the real nodal deformation, we restrict this equivariant
	nodal family to sufficiently small nonnegative real parameters
	\cite[Introduction]{LeviantShustin2018Morsifications}
	\cite[Definitions~1.8 and~1.9]{FominPylyavskyyShustinThurston2022MorsificationsMutations}.
\end{definition}

\begin{definition}
	Let $\mathcal F$ be a real nodal deformation of $(C,0)$, and put
	$f_{s,\R}=F_s|_{\mathbb B_\R}$.  The deformation is a
	\textbf{real morsification} if, for every  $s>0$ small enough, the
	following hold:
	\begin{enumerate}
		\item every critical point of the holomorphic function $F_s$ in
		$\mathbb B$ is real,
		\item every critical point of the real function $f_{s,\R}$ is
			nondegenerate, and
		\item \label{it:only_saddles} the critical points of $f_{s,\R}$ lying
			on its zero level
			$C_{s,\R}$ are precisely its saddle points.
	\end{enumerate}
\end{definition}

\begin{remark}[The convention on the zero level]\label{rem:zero-level-convention}
	The cited definitions of real morsification require all critical points to
	be real and Morse and all saddle points to lie on the zero level
	\cite[Introduction]{LeviantShustin2018Morsifications}
	\cite[Definition~1.10]{FominPylyavskyyShustinThurston2022MorsificationsMutations}.
	In the cited definitions, these conditions still allow a local maximum or
	minimum on the zero level.  For example, the constant real nodal deformation
	of the elliptic node
	\[
		F_s(x,y)=x^2+y^2
	\]
	has one real Morse critical point, a minimum at level zero, and no saddle
	points.  Thus those conditions hold vacuously, even though its zero set has
	an elliptic node.  We remove this loophole in
	condition~\ref{it:only_saddles} by requiring that the critical points on the
	zero level be precisely the saddle points which is equivalent to requiring
	that maxima and minima have positive and negative values respectively.
	With this convention, the
	characterization
	of Leviant and Shustin
	\cite[Remark~4]{LeviantShustin2018Morsifications} states that a
	real nodal
	deformation is a real morsification if and only if, for every
	small
	$s>0$, the curve $C_s$ has precisely
	\[
	\delta(C,0)-\imbr(C,0)
	\]
	real hyperbolic nodes in $\mathbb B$ and no other singularities.  Without
	that
	convention, the above trivial family violates the criterion of Leviant and
	Shustin.

\end{remark}

\subsection*{Divides}

Here we introduce A'Campo's notion of a \textbf{divide} for immersed real curves
in a disk
\cite{ACampo1975MonodromyDeploiement,ACampo1999RealDeformations}.

\begin{definition}
	For a representative $P$ of a reduced real plane curve and a point
	$p\in P_\R$, we say that $p$ is an
	\textbf{ordinary real point of multiplicity $k$} if there are real
	analytic coordinates at $p$ in which a reduced local equation of $P$ is
	\[
	g_1\cdots g_k.
	\]
	Where $g_i\in\R\{x,y\}$, each $g_i = 0$ is smooth at $p$, and the linear 
	parts $dg_i(p)$ define different lines pairwise.  Equivalently, the real 
	curve has
	$k$ smooth local branches through $p$ with pairwise distinct tangent
	lines.
\end{definition}

Our constructions will often produce images of smooth immersions with ordinary
real multiple points that are later perturbed into ordinary double points.

\begin{definition}\label{def:predivide}
	Let $D$ be a closed real disk.  A \textbf{predivide} in $D$ is
	a smooth immersion
	\[
		\alpha:J\to D
	\]
	from a one-dimensional compact manifold such that
	\begin{enumerate}
		\item $\alpha^{-1}(\partial D)=\partial J$, and the restriction of
		$\alpha$ to $\partial J$ is injective.
		\item The image $P = \alpha(J)$ is transverse to $\partial D$.
		\item Every self-intersection is an ordinary multiple point. That is, 
		if $p\in P$ has preimages
		$\alpha^{-1}(p)=\{x_1,\ldots,x_r\}$ with $r\geq 2$,  
		then the tangent lines
		\[
			d\alpha(T_{x_i}J)\subset T_pD,
			\quad i=1,\ldots,r,
		\]
		are pairwise distinct.
	\end{enumerate}
	A \textbf{divide} is a predivide for which every multiple point of $P$ is
	actually an ordinary double point. By abuse of language and notation we
	identify the
	divide $\alpha$ with its image.
\end{definition}

If the closure of $C_{s, \R}$ in $\overline{\mathbb B}_\R$ is a divide, we call 
it the divide of the deformation.  For the
deformations constructed later in \Cref{thm:global-predivide}, the 
normalization of the fiber has
one disk for each real branch and one annulus for each conjugate pair.
For more a more comprehensive introduction to the theory of divides, see the 
references given in the introduction to this work.

	\begin{definition}
		\label{def:pairwise-double-count}
		Let $\alpha:J\to S$ be a smooth immersion of a compact one dimensional
		manifold $J$ into a real surface $S$. Suppose that every multiple point
		of $\alpha$ is ordinary.  Denote $P=\alpha(J)$ and define
		\[
			\dbl(P)
			:=
			\#\bigl\{\{x,y\}\subset J:x\ne y,\ \alpha(x)=\alpha(y)\bigr\}.
		\]
		So an ordinary real point of multiplicity $k$ contributes 
		$\binom{k}{2}$ to $\dbl(P)$.
		If $J_1,J_2\subset J$ are disjoint unions of connected components and
		$P_i=\alpha(J_i)$, we similarly define
		\[
			\dbl(P_1,P_2)
			:=
			\#\bigl\{(x_1,x_2)\in J_1\times J_2:
			\alpha(x_1)=\alpha(x_2)\bigr\}.
		\]
	\end{definition}

	\section{Trace map}
	\label{sec:trace-map}

	\subsection*{Orbits of branches}

	Let $(X, 0)$ be a smooth germ of a surface equipped with a real structure
	$\sigma_X$.	Let $(C,0)\subset (X,0)$ be a reduced real plane curve
		germ, and
	let
	\[
	(C,0)=\bigcup_{B\in\mathcal B}(B,0)
	\]
	be its decomposition into complex branches at the origin. The real structure
	$\sigma_X$ sends each
	branch to its conjugate branch.  We denote the induced
	involution of the set of branches
	$\mathcal B$ by $\sigma$, and set
	\[
	\Omega=\mathcal B/\langle\sigma\rangle .
	\]
	Fix an orbit $\mathfrak o\in\Omega$, and write
	\[
		\mathfrak o=\{B\}
		\qquad\text{or}\qquad
		\mathfrak o=\{Q,\overline Q\},
	\]
	according as it consists of a real branch or a pair of nonreal conjugate
		branches.  Fix an invariant Milnor ball $\mathbb B$ centered at the origin.
			We also use $C,B,Q,\overline Q$, but not $X$, for their representatives
			in $\mathbb B$.

	\subsection*{Compatible parametrizations}

	Choose a disk $\Delta\subset\C$, centered at zero and invariant under
	conjugation.  Let $\Delta_u$ and $\Delta_v$ be two copies of $\Delta$,
	with coordinates $u$ and $v$.  Choose primitive parametrizations
	\[
	\gamma_1:\Delta_u\longrightarrow X,
	\qquad
	\gamma_2:\Delta_v\longrightarrow X
	\]
	as follows:
	\begin{itemize}
		\item For a conjugate pair, choose a primitive parametrization
		$\gamma_Q: \Delta_u\to Q$, and set
		\[
		\gamma_1(u)=\gamma_Q(u),
		\qquad
		\gamma_2(v)=\sigma_X\bigl(\gamma_Q(\overline v)\bigr).
		\]
		Then $\gamma_2$ is a primitive holomorphic parametrization of
		$\overline Q$. We also write $\gamma_{\overline Q}:=\gamma_2$.
		\item For a real branch, choose a primitive parametrization
		$\gamma:\Delta_u\to B$ compatible with complex conjugation, meaning that
		\[
			\sigma_X\bigl(\gamma(u)\bigr)=\gamma(\overline u).
		\]
		Define
		\[
			\gamma_1(u)=\gamma(u),
			\qquad
			\gamma_2(v)=\gamma(v).
		\]
	\end{itemize}
	In both cases we define
	\[
	\boldsymbol\gamma=(\gamma_1,\gamma_2).
	\]

	\subsection*{Trace map}
		Choose a small parameter disk $\Delta_s$ invariant under conjugation.  A
		smoothing of the nodal union of the two normalization disks is
	\begin{equation}\label{eq:canonical-source-representative}
		\mathcal A
		=
		\{(u,v,s)\in\Delta_u\times\Delta_v\times\Delta_s:uv=s\}.
	\end{equation}
	The pair $\boldsymbol\gamma=(\gamma_1,\gamma_2)$ defines the \emph{trace
	map} as follows:
	\begin{equation}\label{eq:trace-map}
		\begin{aligned}
			\widehat\Phi_{\boldsymbol\gamma}:\mathcal A&\longrightarrow X,\\
			(u,v,s)&\longmapsto \gamma_1(u)+\gamma_2(v).
		\end{aligned}
	\end{equation}
	The map takes values in $X$, but its image is not necessarily contained
	in $\mathbb B$.  We restrict its source to
	$\widehat\Phi_{\boldsymbol\gamma}^{-1}(\mathbb B)$ to obtain a map into
	$\mathbb B$.
	Since $\gamma_1(0)=\gamma_2(0)=0$, the restriction of
	$\widehat\Phi_{\boldsymbol\gamma}$ to $s=0$ has an image whose germ is
	$(Q\cup\overline Q,0)$ when the orbit is a conjugate pair.  If the orbit
	consists of a real branch, both disks $\Delta_u$ and $\Delta_v$ map to
	$(B,0)$.

	\subsection*{Quotient symmetry}

	Define the group
	\begin{equation}\label{eq:orbit-automorphism-group}
		G:=
		\begin{cases}
			\{1\}, & \text{if } \mathfrak o=\{Q,\overline Q\},\\
			S_2, & \text{if } \mathfrak o=\{B\}.
		\end{cases}
	\end{equation}
	acting on $\mathcal A$.  If the orbit is a conjugate pair, the action is
	trivial.  If it consists of a real branch,
	the nontrivial element of $S_2$ acts by
	\begin{equation}\label{eq:real-branch-source-swap}
		(u,v,s)\longmapsto(v,u,s).
	\end{equation}
	This swap leaves the trace map invariant because
	$\gamma(u)+\gamma(v)=\gamma(v)+\gamma(u)$.  Thus in both cases
	the trace map $\widehat\Phi_{\boldsymbol\gamma}$ is invariant under $G$.
	Therefore we can define
	\begin{equation}\label{eq:orbit-quotient-source}
		\mathcal S_{\boldsymbol\gamma}:=\mathcal A/G
	\end{equation}
	and obtain the induced map
	\begin{equation}\label{eq:quotient-trace-map}
		\Phi_{\boldsymbol\gamma}:\mathcal
		S_{\boldsymbol\gamma}\longrightarrow X,
		\qquad
		\Phi_{\boldsymbol\gamma}\circ q_{\boldsymbol\gamma}
		=
		\widehat\Phi_{\boldsymbol\gamma},
	\end{equation}
	where $q_{\boldsymbol\gamma}:\mathcal A\to
	\mathcal S_{\boldsymbol\gamma}$ is the quotient map.  We call
	$\Phi_{\boldsymbol\gamma}$ the \emph{quotient trace map}.

	\subsection*{Real structures and equivariance}

	We endow $\mathcal A$ with the real structure (recall
	\Cref{def:real-structure}) given by
	\begin{equation}\label{eq:canonical-source-real-structure}
		\sigma_{\mathcal A}(u,v,s)
		=
		(\overline v,\overline u,\overline s).
	\end{equation}
	It preserves the equation $uv=s$, and
	$\sigma_{\mathcal A}^2=\operatorname{id}_{\mathcal A}$, so it is a real
	structure in the sense of \Cref{def:real-structure}.  The map
	$\sigma_{\mathcal A}$ commutes with the action of $G$, hence descends to
	a real structure
	\[
	\sigma_{\boldsymbol\gamma}:\mathcal S_{\boldsymbol\gamma}\longrightarrow
	\mathcal S_{\boldsymbol\gamma}
	\]
	characterized by
		$\sigma_{\boldsymbol\gamma}\circ q_{\boldsymbol\gamma}
		=
		q_{\boldsymbol\gamma}\circ\sigma_{\mathcal A}.$
	And so $ \sigma_X\circ\widehat\Phi_{\boldsymbol\gamma}
		=
		\widehat\Phi_{\boldsymbol\gamma}\circ\sigma_{\mathcal A}.$
	Therefore the induced map is equivariant, that is,
	\begin{equation}\label{eq:trace-equivariance-down}
		\sigma_X\circ\Phi_{\boldsymbol\gamma}
		=
		\Phi_{\boldsymbol\gamma}\circ\sigma_{\boldsymbol\gamma}.
	\end{equation}

	\subsection*{The fibers}

	Consider the projection
	$\pi_{\mathcal A}:\mathcal A\longrightarrow\Delta_s$, with
	$\pi_{\mathcal A}(u,v,s)=s$, and it descends to
	$\pi_{\boldsymbol\gamma}:\mathcal S_{\boldsymbol\gamma}
	\longrightarrow\Delta_s$.
	For $s \in \Delta_s$, denote
	\[
	\mathcal A_{s}
		=
		\pi_{\mathcal A}^{-1}(s)
		=
	\{(u,v,s)\in\Delta_u\times\Delta_v\times\{s\}:
	uv=s\}.
	\]
	The fiber of $\pi_{\boldsymbol\gamma}$ over $s$ is
	$\mathcal S_{\boldsymbol\gamma,s}
	= \pi_{\boldsymbol\gamma}^{-1}(s) = \mathcal A_{s}/G$.

	\subsection*{The real fibers}

	The central fiber of $\mathcal A$ is the nodal union
	\[
	\mathcal A_0
	=
	(\Delta_u\times\{0\})\cup(\{0\}\times\Delta_v),
	\]
		of two disks, and every nonzero fiber $\mathcal A_{s}$ is an annulus.
		For  $s$ real and small enough, we consider the circle
			\begin{equation}\label{eq:positive-real-source-circle}
				\widehat{\mathcal R}_{s}
				=
				\{(u,v,s)=(\sqrt{s}\,e^{i\theta},\sqrt{s}\,e^{-i\theta},s):
				\theta\in\R/2\pi\Z\}
			\end{equation}
		in the annulus $\mathcal A_{s}$.  This circle is exactly the fixed
		point
		set of $\sigma_{\mathcal A}$ on $\mathcal A_{s}$.  Indeed, each
		point
			of $\widehat{\mathcal R}_{s}$ satisfies $uv=s$ and
		\[
			\overline v=u,\qquad \overline u=v.
		\]
		Since $s \in \R$, the real structure $\sigma_{\mathcal A}$
		satisfies $\sigma_{\mathcal A}(u, v, s) = (u, v, s)$
		on $\widehat{\mathcal R}_{s}$.  Conversely, if
			$(u,v,s)\in\mathcal A_{s}$ is fixed by $\sigma_{\mathcal A}$, then
			$u=\overline v$ and $v = \overline u$.  Since $uv=s$,
			this
			gives that $|u|^2 = s$.  Hence $u=\sqrt{s}\,e^{i\theta}$ and
			$v=\sqrt{s}\,e^{-i\theta}$ for some $\theta\in\R/2\pi\Z$.
			Therefore, we can conclude
		\[
			\operatorname{Fix}
			\bigl(\sigma_{\mathcal A}|_{\mathcal A_{s}}\bigr)
			=
			\widehat{\mathcal R}_{s}.
		\]
		In the quotient of the source we consider the real subset
		$\mathcal R_{\boldsymbol\gamma,s}\subset
		\mathcal S_{\boldsymbol\gamma,s}$ defined by
		\begin{equation}\label{eq:quotient-distinguished-real-set}
		\mathcal R_{\boldsymbol\gamma,s}
		:=
		q_{\boldsymbol\gamma}(\widehat{\mathcal R}_{s})
		\end{equation}
		and study this set in the next two sections.  This subset is always contained
		in the fixed set of $\sigma_{\boldsymbol\gamma}$ on the corresponding
		fiber of $\mathcal S_{\boldsymbol\gamma}$.
		For every $p\in\mathcal R_{\boldsymbol\gamma,s}$, equivariance of
		$\Phi_{\boldsymbol\gamma}$ gives
		\[
			\sigma_X\bigl(\Phi_{\boldsymbol\gamma}(p)\bigr)
			=
			\Phi_{\boldsymbol\gamma}\bigl(\sigma_{\boldsymbol\gamma}(p)\bigr)
			=
			\Phi_{\boldsymbol\gamma}(p).
		\]
		Thus
		\begin{equation}\label{eq:trace-distinguished-real-image}
			\Phi_{\boldsymbol\gamma}(\mathcal R_{\boldsymbol\gamma,s})
			\subset X_\R.
		\end{equation}
		After shrinking $\Delta_s$ if necessary, these images lie in
		$\mathbb B_\R$. The next two sections specialize this construction to 
		the
		two types of orbits of branches.  In the case of a real branch, the
		distinguished
		real subset is an interval in $\mathcal S_{\boldsymbol\gamma,s}$. And for a
		conjugate pair, it remains a circle.

	\section{The case of a real branch}
	\label{sec:chebyshev-normalization}

	We specialize the construction of \Cref{sec:trace-map} to an orbit
	consisting of a real branch, $\mathfrak o=\{B\}$.  Consider a real primitive parametrization
	\[
		\gamma(w)=\bigl(x_\gamma(w),y_\gamma(w)\bigr),
		\quad x_\gamma,y_\gamma\in\R\{w\}.
	\]
	Then $\boldsymbol\gamma=(\gamma,\gamma)$.
	By \Cref{eq:orbit-automorphism-group,eq:real-branch-source-swap},
	$G=S_2$, acting on $\mathcal A$ by exchanging $u$ and $v$.

	By \Cref{eq:canonical-source-representative,eq:orbit-quotient-source}, the
	quotient of the source is
	$\mathcal S_{\boldsymbol\gamma} = \mathcal A/S_2$.
	The functions $t=u+v$ and $s=uv$ form a coordinate system on this
	quotient:
	\[
		\bigl(\C\{u,v,s\}/(uv-s)\bigr)^{S_2}
		=
		\C\{u+v,s\}
		=
		\C\{t,s\}.
	\]
	After shrinking the representatives if necessary, we have 
	$
		\mathcal S_{\boldsymbol\gamma}
		\simeq
		\Delta_t\times\Delta_s.
	$
	In these coordinates, the quotient map and the real structure induced by
	\Cref{eq:canonical-source-real-structure} are given by
	\[
		q_{\boldsymbol\gamma}(u,v,s)=(u+v,s),
		\quad \text{ and }
		\sigma_{\boldsymbol\gamma}(t,s)
		=
		(\overline t,\overline s),
	\]
	respectively. In particular, every fiber
	$\mathcal S_{\boldsymbol\gamma,s}$ is homeomorphic to a disk. By
	\Cref{eq:trace-map,eq:quotient-trace-map}, we have
	\[
		\Phi_{\boldsymbol\gamma}(u+v,s)
		=
		\gamma(u)+\gamma(v),
		\quad\enspace (u,v,s)\in\mathcal A.
	\]
	At $s = 0$, taking $(u,v)=(t,0)$ gives
	\[
		\Phi_{\boldsymbol\gamma}(t,0)=\gamma(t).
	\]
	This means that the restriction of $\Phi_{\boldsymbol\gamma}$ to the 
	central fiber of
	$\mathcal S_{\boldsymbol\gamma}$ is the chosen normalization map of $B$.

	Let $0<s=\rho^2\ll1$.  On the circle fixed by the real structure (recall
	\Cref{eq:positive-real-source-circle}), we have
	\[
		u=\rho e^{i\theta},
		\qquad
		v=\rho e^{-i\theta},
		\qquad
		t=u+v=2\rho\cos\theta.
	\]
	Hence \Cref{eq:quotient-distinguished-real-set} gives
	\[
		\mathcal R_{\boldsymbol\gamma,s}
		=
		[-2\rho,2\rho]\times\{s\}.
	\]
		This interval is properly contained in the locus of the
		corresponding fiber of $\mathcal S_{\boldsymbol\gamma}$ fixed by the real
		structure $\sigma_{\boldsymbol\gamma}$, namely
	\[
		\operatorname{Fix}
		\bigl(
			\sigma_{\boldsymbol\gamma}
			|_{\mathcal S_{\boldsymbol\gamma,s}}
		\bigr)
		=
		(\Delta_t\cap\R)\times\{s\}.
	\]
	By \Cref{eq:trace-equivariance-down}, $\Phi_{\boldsymbol\gamma}$ maps
	this whole locus fixed by the real structure into $X_\R$. We renormalize by 
	setting
	\[
		r=2\rho,
		\qquad
		s=\frac{r^2}{4}.
	\]
		The  interval above  is then $[-r,r]$.  For $0 < r \ll 1$, use
		$t$ as coordinate on the locus of the corresponding fiber
		$\mathcal S_{\boldsymbol\gamma, r^2/4}$ of $\mathcal S_{\boldsymbol\gamma}$
		fixed by the real structure
		$\sigma_{\boldsymbol\gamma}$, and define
	\begin{equation}\label{eq:real-branch-interval-map}
	\begin{aligned}
		\alpharb{r}:
		\operatorname{Fix}
		\bigl(
			\sigma_{\boldsymbol\gamma}
			|_{\mathcal S_{\boldsymbol\gamma,r^2/4}}
		\bigr)
		&\longrightarrow X_\R,\\
		t&\longmapsto
		\Phi_{\boldsymbol\gamma}
		\left(t,\frac{r^2}{4}\right).
	\end{aligned}
	\end{equation}

	We next express the quotient trace map in invariant coordinates using
	Dickson polynomials.  For a parametrization in real Newton--Puiseux form,
	this will lead to the desired predivide. Write
	\[
		x_\gamma(w)=\sum_{n\ge1}\alpha_nw^n,
		\qquad
		y_\gamma(w)=\sum_{n\ge1}\beta_nw^n,
		\qquad
		\alpha_n,\beta_n\in\R.
	\]
	Then, for computating 
	$\Phi_{\boldsymbol\gamma}$ we just need to express the power sums $u^n+v^n$ 
	in
	terms of $t=u+v$ and $s=uv$. Define the Dickson polynomials
	$D_n(t,s)$ of the first kind recursively
	by
	\[
		D_0=2,
		\qquad
		D_1=t,
		\qquad
			D_{n+1}=tD_n-sD_{n-1},
			\qquad n\ge1.
	\]
	The power sums have the same initial values and satisfy the same
	recurrence, since
	\[
		u^{n+1}+v^{n+1}
		=
		(u+v)(u^n+v^n)-uv(u^{n-1}+v^{n-1}).
	\]
By induction, we get
	\begin{equation}\label{eq:dickson-defining-property}
			D_n(u+v,uv)=u^n+v^n,
			\qquad n\ge0.
	\end{equation}
	See \cite[Chapter 1]{LidlMullenTurnwald1993Dickson} for further
	properties of these polynomials. Applying
	\Cref{eq:dickson-defining-property} term by term gives
	\begin{equation}\label{eq:real-branch-trace-power-series}
		\Phi_{\boldsymbol\gamma}(t,s)
		=
		\left(
			\sum_{n\ge1}\alpha_nD_n(t,s),
			\sum_{n\ge1}\beta_nD_n(t,s)
		\right).
	\end{equation}
	On the fiber over $s=r^2/4$, we can look at the Dickson polynomials as
	{\em rescaled} Chebyshev polynomials.  Let $T_n$ denote the Chebyshev
	polynomial of the first kind. These are
	characterized by
	\[
		T_n(\cos\theta)=\cos(n\theta).
	\]

	\begin{proposition}
		\label{prop:chebyshev-normalization}
		For every $n\ge0$ and every $r>0$, the identity
		\begin{equation}\label{eq:chebyshev-trace}
			D_n\left(t,\frac{r^2}{4}\right)
			=
			2^{1-n}r^nT_n(t/r).
		\end{equation}
		holds in $\R[t]$.
		\end{proposition}

		\begin{proof}
			Fix $r>0$.  For $\theta\in\R$, set
			\[
				u=\frac r2e^{i\theta},
				\qquad
				v=\frac r2e^{-i\theta}.
			\]
			Then $uv=r^2/4$ and $u+v=r\cos\theta$.  By
			\Cref{eq:dickson-defining-property},
			\[
				D_n\left(r\cos\theta,\frac{r^2}{4}\right)
				=
				u^n+v^n
				=
				2\left(\frac r2\right)^n\cos(n\theta)
				=
				2^{1-n}r^nT_n(\cos\theta).
			\]
			Equation~\eqref{eq:chebyshev-trace} therefore holds for every
			$t=r\cos\theta\in[-r,r]$. And since this is an identity between 
			polynomials in $t$ that holds on an interval, it must hold for all 
			$t$.
		\end{proof}

		\begin{corollary}
			\label{cor:real-branch-trace-predivide}
			Assume that the chosen parametrization for the orbit consisting of a real
			branch is in real Newton--Puiseux form
			$\gamma(w)=(w^m,\sum_{j>m}a_jw^j)$.  With $s=r^2/4$ as in
			\Cref{eq:real-branch-interval-map}, reparametrize the interval
			$[-r,r]$ by $t=rx$.  Then
			\[
				\alpharb{r}(rx)
				=
				\left(
				2^{1-m}r^mT_m(x),\,
				\sum_{j>m}a_j2^{1-j}r^jT_j(x)
				\right).
			\]
			Let $\mathbb B'$ be any  Milnor ball for $B$
			invariant under conjugation, and put
			$\mathbb B'_\R:=\mathbb B'\cap X_\R$. Then, there is $r_0>0$ such 
			that, 
			for
			every $0<r<r_0$, if $I_r^\circ$ is the component of
			$\{t \in \Delta_t\cap\R: \alpharb{r}(t) \in \mathbb B'_\R\}$  
			containing
			$[-r,r]$, and $I_r$
			is its closure, then
			\[
				\alpharb{r}|_{I_r}:I_r\longrightarrow\overline{\mathbb B'}_\R
			\]
			parametrizes a predivide.  Denote its image by
			$\Gamma_{\boldsymbol\gamma,r}^{\mathrm{rb},\overline{\mathbb B'}_\R}
			:=\alpharb{r}(I_r)$.  Then
			\[
				\dbl(
				\Gamma_{\boldsymbol\gamma,r}^{\mathrm{rb},\overline{\mathbb B'}_\R})
				=
				\delta(B,0).
			\]
		\end{corollary}

			\begin{proof}
				The displayed formula for $\alpharb{r}$ follows by combining
					\Cref{eq:real-branch-trace-power-series} with
					\Cref{prop:chebyshev-normalization}.  The statement about 
					the
					predivide and the equality for $\dbl$ follow from Gusein-Zade's
					construction using Chebyshev polynomials for a real branch in
					Newton--Puiseux form
					\cite[Section 5]{GuseinZade1974DynkinDiagrams}.
					See also \cite[Section 5.3]{ACampoPortilla2025Divides}.
					The endpoints of $I_r$ are sent to 
					$\partial\overline{\mathbb B'}_\R$, and no
					interior point does.  Thus $\alpharb{r}|_{I_r}$ parametrizes a
					predivide in the closed disk $\overline{\mathbb B'}_\R$.

				\end{proof}

	\begin{remark}
		The construction above using Chebyshev polynomials is not used in proof 
		of the main theorem in
		\Cref{sec:global-predivide}.  Its role is to identify the specialization of
		the trace map to a real branch and show that our construction
		generalizes that of Gusein-Zade in
		\cite{GuseinZade1974IntersectionMatrices} and makes Chebyshev
		polynomials appear naturally.  In the final construction, real branches
		are
		treated by A'Campo's method of translations and contractions.
	\end{remark}

	\section{The case of a conjugate pair with distinct tangents}
	\label{sec:conjugate-pair-case}

	We apply the construction of \Cref{sec:trace-map} to an orbit consisting
	of a conjugate pair
	\[
		\mathfrak o=\{Q,\overline Q\},
	\]
	and assume that $Q$ and $\overline Q$ have distinct nonreal tangent
	lines (recall \Cref{def:real-tangent}).  Put $m=\mt(Q,0)$, and let
	$L = T_0Q$.  Since the tangent lines are distinct,
	$T_0\overline Q=\sigma_X(L)\ne L$.  Choose a complex linear form $q$
	whose kernel is $L$, and define
	\begin{equation}\label{eq:distinct-conjugate-adapted-p-coordinate}
		p(z)=\overline{q(\sigma_X(z))}.
	\end{equation}
	Then $\ker p=\sigma_X(L)$.  Hence $p$ and $q$ are linearly
	independent, and $(p,q)$ is a complex linear coordinate system.  With
	respect to this coordinate system, the real structure $\sigma_X$ is
	given by
	\begin{equation}\label{eq:distinct-conjugate-real-structure}
		(p(z),q(z))\longmapsto(\overline{q(z)},\overline{p(z)}), \quad
	 z\in \C^2
	\end{equation}
	and the real plane is the locus $\{q=\overline p\}$.  Since
	$p|_L\not\equiv0$, the order of $p$ along $Q$ is $m$.  After changing
	a primitive parameter, we may therefore write
	\begin{equation}\label{eq:distinct-conjugate-np-parametrization}
		\gamma_Q(u)=(u^m,\varphi(u)),
		\qquad
		\varphi(u)=\sum_{k>m}a_ku^k.
	\end{equation}
	Since $\ker q=T_0Q$, the series
	$\varphi=q\circ\gamma_Q$ has order greater than $m$.  For $\overline Q$, we 
	use the conjugate parametrization
	\begin{equation}\label{eq:distinct-conjugate-conjugate-parametrization}
		\gamma_{\overline Q}(v)
		=
		\sigma_X\bigl(\gamma_Q(\overline v)\bigr)
		=
		(\overline\varphi(v),v^m),
		\qquad
		\overline\varphi(v)=\sum_{k>m}\overline{a_k}v^k.
	\end{equation}
	Let $\gamma_1 = \gamma_Q$, let $\gamma_2=\gamma_{\overline Q}$, and put
	$\boldsymbol\gamma=(\gamma_1,\gamma_2)$.
	Since $G=\{1\}$ by \Cref{eq:orbit-automorphism-group},
	\Cref{eq:orbit-quotient-source,eq:quotient-distinguished-real-set} we have 
	that
	\[
		\mathcal S_{\boldsymbol\gamma}=\mathcal A
		\quad \text{ and }\quad
		\mathcal R_{\boldsymbol\gamma,s}=\widehat{\mathcal R}_s.
	\]
	For $s>0$, this is the fixed locus of the annulus $\mathcal A_s$ by
	\Cref{eq:positive-real-source-circle}.  Unlike the case of a real branch,
	$\mathcal R_{\boldsymbol\gamma,s}$ remains a circle. Here we take
	$s=r^2>0$.  By \Cref{eq:trace-distinguished-real-image}, the restriction
	of the trace map to this fixed circle takes values in $X_\R$
	and defines
	\begin{equation}\label{eq:conjugate-pair-circle-map}
	\begin{aligned}
		\alphacp{r}:\R/2\pi\Z&\longrightarrow X_\R,\\
		\theta&\longmapsto
		\gamma_Q(re^{i\theta})
		+
		\gamma_{\overline Q}(re^{-i\theta}).
	\end{aligned}
	\end{equation}
		We denote its image by
		\begin{equation}\label{eq:predivide_conjugate}
			\Gamma^{\mathrm{cp}}_{\boldsymbol\gamma,r}
			:=
			\alphacp{r}(\R/2\pi\Z).
		\end{equation}

	The construction of \Cref{sec:trace-map} allows, a priori, arbitrary
	primitive parametrizations, but changing them changes $\boldsymbol\gamma$
	and
	$\Phi_{\boldsymbol\gamma}$.  It may also change the number of pairs of
	distinct parameters with the same image, as
	\Cref{ex:adapted-source-coordinate-needed} shows.
	Substituting the adapted parametrizations
	\Cref{eq:distinct-conjugate-np-parametrization,eq:distinct-conjugate-conjugate-parametrization}
	into \Cref{eq:conjugate-pair-circle-map} gives
	\begin{equation}\label{eq:conjugate-circle-map-coordinates}
	\alphacp{r}(\theta)
	=
	\left(
		r^me^{im\theta}
		+
		\sum_{n>m}\overline{a_n}r^ne^{-in\theta},
		\,
		r^me^{-im\theta}
		+
		\sum_{n>m}a_nr^ne^{in\theta}
	\right).
	\end{equation}

		\begin{theorem}
			\label{thm:distinct-tangent-conjugate-pair-predivide}
			The expression for
			$\alphacp{r}(\theta)$ depends real analytically on $(r, \theta)$.
			Also, there are $r_0>0$ and a closed real disk
			$D \subset \mathbb B_\R$, with $0\in\operatorname{int}D$, such
				that, for $0<r<r_0$, the image
				$\Gamma^{\mathrm{cp}}_{\boldsymbol\gamma,r}$ defined in
				\Cref{eq:predivide_conjugate} is contained in
			$\operatorname{int}D$ and is a predivide in $D$ in the sense of
			\Cref{def:predivide}.  Moreover,
			\[
				\dbl(\Gamma^{\mathrm{cp}}_{\boldsymbol\gamma,r})
				=
			2\delta(Q,0)+(Q\cdot\overline Q)_0-1
			=
			\delta(Q\cup\overline Q,0)-1.
		\]
	\end{theorem}

	We first prove some preliminary results needed in the proof of
	\Cref{thm:distinct-tangent-conjugate-pair-predivide}.  After factoring
	$r^m$ in \Cref{eq:conjugate-circle-map-coordinates}, the leading map is
	$\theta\mapsto e^{im\theta}$.  For two distinct parameters $\theta$ and
	$\eta$, this leading map takes the same value precisely when
		$e^{i(\eta-\theta)}$ is a solution $\zeta\ne1$ of $\zeta^m=1$.  Let
		\[
			\beta_1<\cdots<\beta_g
		\]
		be the characteristic exponents of the Newton--Puiseux series
		$\varphi$ in
		\Cref{eq:distinct-conjugate-np-parametrization}.  Put
		\begin{equation}\label{eq:characteristic-gcd-sequence}
			d_0=m,
			\qquad
			d_j=\gcd(m,\beta_1,\ldots,\beta_j),
			\qquad 1\le j\le g.
		\end{equation}
		So $\beta_j$ is the smallest exponent $n$ with
		$a_n\ne0$ that is not divisible by $d_{j-1}$.  Since the
		parametrization is primitive,
		$d_g = 1$. Note that $g=0$ only in the smooth case when $m=1$.  For each
		solution  of
		$\zeta^m=1$ different from $1$, define
		\begin{equation}\label{eq:kappa_order}
		\kappa(\zeta)
		=
		\ord_u\bigl(\varphi(u)-\varphi(\zeta u)\bigr).
	\end{equation}
	The order above in \cref{eq:kappa_order} has to be finite. Indeed, if it 
	was not finite, then
	$\gamma_Q(u)=\gamma_Q(\zeta u)$, which contradicts the primitivity of the
	parametrization.

\begin{lemma}
	\label{lem:root-of-unity-separation-sum}
		Let $Q$ be a complex plane branch equipped with a Newton--Puiseux 
		parametrization
		\Cref{eq:distinct-conjugate-np-parametrization}. Define
		$\beta_j$, $d_j$, and $\kappa(\zeta)$ as above.
		For $1\le j\le g$, if $\zeta^{d_{j-1}} = 1$ and
		$\zeta^{d_j}\ne1$, then $\kappa(\zeta)=\beta_j$.  Moreover,
		\[
		\sum_{\substack{\zeta^m=1\\ \zeta\ne1}}\kappa(\zeta)
	=
	\sum_{j=1}^g(d_{j-1}-d_j)\beta_j
	=
	2\delta(Q,0)+m-1.
	\]
\end{lemma}

\begin{proof}
	By the definition of the characteristic exponents, every exponent
	$n<\beta_j$ with $a_n\ne0$ is divisible by $d_{j-1}$, but
	$a_{\beta_j}\ne0$ and $\beta_j$ is not divisible by $d_{j-1}$.  Hence, if
	$\zeta^{d_{j-1}}=1$, all terms of
	$\varphi(u)-\varphi(\zeta u)$ of order $<\beta_j$ vanish.  If also
		$\zeta^{d_j} \ne 1$, then $\zeta^{\beta_j}\ne1$, so the term indexed by
		$\beta_j$ does not vanish:
	\[
		a_{\beta_j}(1-\zeta^{\beta_j})u^{\beta_j}\ne0.
	\]
		So, by definition, $\kappa(\zeta) = \beta_j$.  Write
		$\mu_e=\{\zeta\in\C:\zeta^e=1\}$.  The roots of unity
	$\zeta^{d_{j-1}}=1$ with $\zeta^{d_j}\ne 1$, are exactly
	$\mu_{d_{j-1}}\setminus\mu_{d_j}$, so there are $d_{j-1}-d_j$ of
	them.  This proves the first equality in the displayed formula.

	Wall, in \cite[Corollary~4.3.7 and the computation after Example
	4.3.1]{Wall2004SingularPoints},  proves that, for the semigroup $S(Q)$
	of the branch,
	\[
		2\delta(Q,0)=N(S(Q))+1
	\]
	with
	\[
		N(S(Q))
		=
		\sum_{j=1}^g(d_{j-1}-d_j)(\beta_j-1)-1.
	\]
   Here $e_j$ in Wall's notation is our $d_j$.  Therefore,
	\[
		2\delta(Q,0)
		=
		\sum_{j=1}^g(d_{j-1}-d_j)(\beta_j-1).
	\]
	Since $d_0=m$ and $d_g=1$, we get
	$
		\sum_{j=1}^g(d_{j-1}-d_j)=m-1.
	$
	And finally, putting the three equations above together, we find
	\[
		\sum_{j=1}^g(d_{j-1}-d_j)\beta_j
		=
		2\delta(Q,0)+m-1.
	\]
\end{proof}

	To study the circle map $\alphacp{r}$ in
	\Cref{eq:conjugate-pair-circle-map}, for $r > 0$ we use the function
		\begin{equation}\label{eq:conjugate-zr-definition}
		z_r(\theta)
		=
		e^{im\theta}
		+
		\sum_{n>m}\overline{a_n}r^{n-m}e^{-in\theta}.
		\end{equation}
		Writing explicitly the common factor $r^m$ in
		\Cref{eq:conjugate-circle-map-coordinates}, we get
		\begin{equation}\label{eq:conjugate-alpha-zr-reduction}
		\alphacp{r}(\theta)
		=
		\left(r^mz_r(\theta),r^m\overline{z_r(\theta)}\right).
		\end{equation}
		By \Cref{eq:distinct-conjugate-real-structure},
		$X_\R=\{q=\overline p\}$.  For fixed $r>0$, define
		\begin{equation}\label{eq:conjugate-Lr-definition}
		\begin{aligned}
			L_r:\C&\longrightarrow X_\R,\\
			z&\longmapsto (r^mz,r^m\overline z).
		\end{aligned}
			\end{equation}
			Then $L_r$ is a real linear isomorphism and
			$\alphacp{r}=L_r\circ z_r$.

			\begin{remark}
				\label{rem:linear-reduction-to-zr}
				By
				\Cref{eq:conjugate-alpha-zr-reduction,eq:conjugate-Lr-definition},
				two parameter values have the
					same image under $\alphacp{r}$ if and only if they have the same
						image under $z_r$.  The map $L_r$ also preserves immersions, the transversality
						of the corresponding intersections, ordinary multiple points, and the
						value of $\dbl$.
					Thus these properties for $\alphacp{r}$ can be
					checked on $z_r$.
				\end{remark}

					Put $\mathbb T=\R/2\pi\Z$.  For $[t]\in\mathbb T$, we
					denote its distance to $[0]$ by
					\[
					\begin{aligned}
						|\cdot|_{\mathbb T}:\mathbb T&\longrightarrow\R_{\ge0},\\
						[t]&\longmapsto\min_{k\in\Z}|t+2\pi k|.
					\end{aligned}
					\]

					\begin{lemma}
					\label{lem:distinct-basic-estimates}
					The maps $\alphacp{r}$ form a real analytic
					family parametrized by $r$.  There are $C>0$,
					$\epsilon > 0$, and $r_0>0$ such that, for every
					$0<r<r_0$, the map $z_r$, and hence
					$\alphacp{r}$, is an immersion, and
					$z_r(\theta)\ne z_r(\eta)$ whenever
					$0<|\theta-\eta|_{\mathbb T}<\epsilon$, for all
					$\theta, \eta \in \R/2\pi\Z$.  Moreover,
					\begin{equation}\label{eq:conjugate-zr-uniform-bound}
						\sup_\theta |z_r(\theta)|\le C,
					\end{equation}
					and
					\begin{equation}\label{eq:conjugate-alpha-uniform-smallness}
						\sup_\theta |\alphacp{r}(\theta)|\le Cr^m.
					\end{equation}
			\end{lemma}

				\begin{proof}
					The expression
					\[
						r^mz_r(\theta)
						=
						r^me^{im\theta}
						+
						\sum_{n>m}\overline{a_n}r^ne^{-in\theta}
					\]
					is convergent for $|r|$ sufficiently small by analyticity of Puiseux series, so
					$(r, \theta)\mapsto\alphacp{r}(\theta)$
					is real analytic by
					\Cref{eq:conjugate-alpha-zr-reduction}.
					Differentiating $z_r$, we get
						\[
							z_r'(\theta)
							=
							ime^{im\theta}
							+
							\sum_{n>m}(-in)\overline{a_n}r^{n-m}e^{-in\theta}.
						\]
						The sum in the above expression is convergent for small
						$r$ (by analyticity of Puiseux series), and every term
						contains a positive power of $r$.  Hence it is uniformly
						small as $r\to0$.  Since $ime^{im\theta}$ never
						vanishes, $z_r$ is an immersion for every
						$r>0$ small enough.  By 
						\Cref{rem:linear-reduction-to-zr}, the same
					holds for $\alphacp{r}$.

						The estimate above for the derivative of $z_r$ also implies uniform local injectivity.
						Since the parameter circle is compact and
						$ime^{im\theta}$ never vanishes, there are $c>0$ and
						$\epsilon>0$
						such that, whenever
						\mbox{$0<|\theta-\eta|_{\mathbb T}<\epsilon$},
					\[
						|e^{im\theta}-e^{im\eta}|
						\ge c|\theta-\eta|_{\mathbb T}.
					\]
						Let $\mu(r)\to0$ be a uniform bound for the derivative of
						$z_r(\theta)-e^{im\theta}$.  The fundamental theorem of
						calculus gives
						\[
							\big|(z_r(\theta)-z_r(\eta))
							-(e^{im\theta}-e^{im\eta})\big|
							\le \mu(r)|\theta-\eta|_{\mathbb T}
						\]
						uniformly for
						$0<|\theta-\eta|_{\mathbb T}<\epsilon$.  After decreasing
						$r_0$ so that $\mu(r)<c$ for $0<r<r_0$, the reverse
						triangle inequality gives
						$z_r(\theta)\ne z_r(\eta)$ for all $0<r<r_0$ and
						$0<|\theta-\eta|_{\mathbb T}<\epsilon$.

					The uniform bound
					\Cref{eq:conjugate-zr-uniform-bound} follows from
					\Cref{eq:conjugate-zr-definition} and convergence of the
					Puiseux series.  Together with
					\Cref{eq:conjugate-alpha-zr-reduction}, this gives
					\Cref{eq:conjugate-alpha-uniform-smallness}.
				\end{proof}

			Next, we count the unordered pairs $\{\theta,\eta\}$, with
			$\theta\ne\eta$, such that $z_r(\theta) = z_r(\eta)$ and show
			that every multiple point is an ordinary point of multiplicity $k$. 
			This is the crucial ingredient for the proof of 
			\Cref{thm:distinct-tangent-conjugate-pair-predivide} right after.
			\begin{proposition}
			\label{prop:direct-root-of-unity-double-count}
			We use the hypotheses and notation fixed at the start of
			\Cref{sec:conjugate-pair-case}. Let
			\[
			\begin{aligned}
				z_r:\R/2\pi\Z&\longrightarrow\C,\\
				\theta&\longmapsto
				e^{im\theta}
				+
				\sum_{n>m}\overline{a_n}r^{n-m}e^{-in\theta}.
			\end{aligned}
			\]
			For every $r>0$ small enough, this map is an immersion whose
			multiple points are ordinary.  With $\kappa(\zeta)$ as defined in
			\Cref{eq:kappa_order},
			\[
				\dbl\bigl(z_r(\R/2\pi\Z)\bigr)
				=
				\sum_{\substack{\zeta^m=1\\ \zeta\ne1}}
				\bigl(m+\kappa(\zeta)\bigr)\\
				=
				2\delta(Q,0)+m^2-1.
			\]
			\end{proposition}

			\begin{proof}
				Since this is a long proof, we organize it in small sections.
				\smallskip
				
				\noindent\emph{Localization in sectors.}
					By \Cref{lem:distinct-basic-estimates}, choose
					$\epsilon>0$ and $r_0>0$ such that, for every
				$0<r<r_0$, the map $z_r$ is an immersion and
					\[
						0<|\theta-\eta|_{\mathbb T}<\epsilon
						\quad\Longrightarrow\quad
						z_r(\theta)\ne z_r(\eta)
					\]
					for all $\theta,\eta\in\R/2\pi\Z$.
					Set
					\[
					\begin{aligned}
						F_r:(\R/2\pi\Z)\times(\R/2\pi\Z)&\longrightarrow\C,\\
						(\theta,\eta)&\longmapsto z_r(\theta)-z_r(\eta).
					\end{aligned}
					\]
					Write $z_r(\theta)=e^{im\theta}+T_r(\theta)$, where
					\[
					\begin{aligned}
						T_r:\R/2\pi\Z&\longrightarrow\C,\\
						\theta&\longmapsto
						\sum_{n>m}\overline{a_n}r^{n-m}e^{-in\theta}.
					\end{aligned}
					\]
					Since the Puiseux series is convergent, $T_r\to0$ uniformly on
					$\R/2\pi\Z$ as $r\to0$.
				Choose $0<\varepsilon_{\mathrm{sec}}<\epsilon$ so small that
				the {\em sectors}
				\[
					\mathcal S_j(\varepsilon_{\mathrm{sec}})
					=
					\left\{(\theta,\eta):
					|\eta-\theta-\alpha_j|_{\mathbb T}<\varepsilon_{\mathrm{sec}}\right\},
					\qquad
					\alpha_j=\frac{2\pi j}{m},\quad j=0,\ldots,m-1,
					\]
					are pairwise disjoint in
					$(\R/2\pi\Z)\times(\R/2\pi\Z)$.  If
					$F_r(\theta,\eta)=0$, then
					\[
						e^{im\theta}-e^{im\eta}
						=-(T_r(\theta)-T_r(\eta)),
					\]
					and therefore
					\[
						|1-e^{im(\eta-\theta)}|
						=
						|T_r(\theta)-T_r(\eta)|
						\le 2\|T_r\|_\infty
						\longrightarrow0
						\quad\enspace (r\to0).
					\]
					This convergence is uniform over all such solutions
					$(\theta,\eta)$.
					Since the zeros of $e^{im\beta}-1$ on $\R/2\pi\Z$ are
				precisely $\alpha_j$ for $j=0,\ldots,m-1$, then
				after decreasing $r_0$, every solution $(\theta,\eta)$ of
				$F_r(\theta,\eta)=0$ lies in one
				of the sectors $\mathcal S_j(\varepsilon_{\mathrm{sec}})$.  The sector
				$\mathcal S_0(\varepsilon_{\mathrm{sec}})$ contains no solutions away from the
				diagonal
				by \Cref{lem:distinct-basic-estimates}.  So all the
				solutions away from the diagonal lie in the nontrivial sectors
				$\mathcal S_j(\varepsilon_{\mathrm{sec}})$, with $1\le j\le m-1$.
				\Cref{fig:torus-sectors-m6} illustrates two of these sectors in 
				a fundamental domain of the parameter torus.

				\begin{figure}[htbp]
					\centering
					\tikzsetnextfilename{torus-sectors-m6}
					\scalebox{0.72}{\begin{tikzpicture}[
   x=6.8cm,
   y=6.8cm,
   line cap=butt,
   line join=round
]
\definecolor{sectorzero}{HTML}{2F80ED}
\definecolor{sectortwo}{HTML}{F2994A}

\fill[gray!4] (0,0) rectangle (1,1);

\begin{scope}
   \clip (0,0) rectangle (1,1);

   \draw[sectortwo,opacity=.35,line width=18pt]
      (-.15,.183333) -- (1.15,1.483333);
   \draw[sectortwo,opacity=.35,line width=18pt]
      (-.15,-.816667) -- (1.15,.483333);

   \draw[sectorzero,opacity=.35,line width=18pt]
      (-.15,-.15) -- (1.15,1.15);
   \draw[sectorzero,opacity=.35,line width=18pt]
      (-.15,.85) -- (.15,1.15);
   \draw[sectorzero,opacity=.35,line width=18pt]
      (.85,-.15) -- (1.15,.15);

   \draw[sectorzero,line width=.9pt]
      (-.15,-.15) -- (1.15,1.15);
   \draw[sectorzero,line width=.9pt]
      (-.15,.85) -- (.15,1.15);
   \draw[sectorzero,line width=.9pt]
      (.85,-.15) -- (1.15,.15);
   \draw[sectortwo,line width=.9pt]
      (-.15,.183333) -- (1.15,1.483333);
   \draw[sectortwo,line width=.9pt]
      (-.15,-.816667) -- (1.15,.483333);
\end{scope}

\draw[black,line width=.7pt] (0,0) rectangle (1,1);

\node[below left] at (0,0) {\(\scriptstyle(0,0)\)};
\node[below] at (1,0) {\(\scriptstyle 2\pi\)};
\node[left] at (0,1) {\(\scriptstyle 2\pi\)};
\node[below] at (.5,-.11) {\(\theta\)};
\node[left] at (-.11,.5) {\(\eta\)};

\begin{scope}[shift={(1.18,.66)}]
   \fill[sectorzero,opacity=.35] (0,0) rectangle (.16,.065);
   \draw[sectorzero,line width=.9pt] (0,.0325) -- (.16,.0325);
   \node[right] at (.21,.0325)
      {\(\mathcal S_0(\varepsilon_{\mathrm{sec}})\)};

   \fill[sectortwo,opacity=.35] (0,-.13) rectangle (.16,-.065);
   \draw[sectortwo,line width=.9pt] (0,-.0975) -- (.16,-.0975);
   \node[right] at (.21,-.0975)
      {\(\mathcal S_2(\varepsilon_{\mathrm{sec}})\), \(m=6\)};
\end{scope}
\end{tikzpicture}}
					\caption{Two of the sectors
					$\mathcal S_j(\varepsilon_{\mathrm{sec}})$ in a square fundamental
					domain of $(\R/2\pi\Z)^2$, illustrated for $m=6$.}
					\label{fig:torus-sectors-m6}
				\end{figure}

				\smallskip
				\noindent\emph{Intersection points in a sector.}
					Fix $j\in\{1,\ldots,m-1\}$, put
					$\alpha=\alpha_j$, $\zeta=e^{i\alpha}$, and put
					$\kappa = \kappa(\zeta)$.  By \Cref{eq:kappa_order}, $\kappa$ is
					the smallest exponent $n$ with $a_n\ne0$ and
				$\zeta^n\ne1$.  So, $a_n(1-\zeta^n)=0$ for 
				$m<n<\kappa$, and $a_\kappa(1-\zeta^\kappa)\ne 0$.
				We study the zeros of $F_r$ in $\mathcal S_j(\varepsilon_{\mathrm{sec}})$.
				There we can write
				\[
					\eta=\theta+\alpha+h
				\]
				with $h$ small.  Using $\zeta^m=1$, we have
				\begin{equation}\label{eq:F_r_estimate}
				\begin{aligned}
					F_r(\theta,\theta+\alpha+h)
					&=e^{im\theta}(1-e^{imh})\\
					&\quad+
					\sum_{n>m}\overline{a_n}r^{n-m}e^{-in\theta}
					(1-\zeta^{-n}e^{-inh}).
				\end{aligned}
				\end{equation}
				For $n < \kappa$ with $a_n\ne0$, one has
				$\zeta^n=1$, so
				$1-\zeta^{-n}e^{-inh}=1-e^{-inh}=O(h)$.
				Thus the terms corresponding to $n<\kappa$ contribute
				$O(r|h|)$. And the terms with $n\ge\kappa$
				contribute $O(r^{\kappa-m})$.  Also
				\[
					e^{im\theta}(1-e^{imh})
					=
					-imh e^{im\theta}+O(h^2).
				\]
				Hence \Cref{eq:F_r_estimate} gives the expansion
				\begin{equation}\label{eq:F_r_sector_asymptotic}
					F_r(\theta,\theta+\alpha+h)
					=
					-imh e^{im\theta}
					+O(h^2)
					+O(r|h|)
					+O(r^{\kappa-m}).
				\end{equation}
					In \Cref{eq:F_r_sector_asymptotic}, after decreasing
				$\varepsilon_{\mathrm{sec}}$, and then $r_0$, the sum of the terms involving
				$h$ has modulus bounded below by a positive
				multiple of $|h|$.  Therefore any zero
				$(\theta,\theta+\alpha+h)$ of $F_r$ in
				$\mathcal S_j(\varepsilon_{\mathrm{sec}})$ satisfies
				\begin{equation}\label{eq:sector-displacement-bound}
					|h|\le Cr^{\kappa-m}.
				\end{equation}
				Thus every zero in $\mathcal S_j(\varepsilon_{\mathrm{sec}})$ has the form
				\[
					\eta=\theta+\alpha+r^{\kappa-m}x
				\]
				with $x$ bounded, and $x \in \R$ because $h\in\R$.

					\smallskip
					\noindent\emph{Stability of zeros.}
					The bound above shows that every zero in this sector has
					$h=r^{\kappa-m}x$, with $x$ bounded independently of $r$.  We
					therefore compare the rescaled equation with its limit in $C^1$ on a
					fixed compact cylinder. We carry out this intricate 
					analysis in 
					4 steps.

					\smallskip
					\noindent\emph{Step 1: the limiting equation.}
				Motivated by \Cref{eq:sector-displacement-bound}, define
					\[
					\begin{aligned}
						F_{r,\zeta}:(\R/2\pi\Z)\times\R&\longrightarrow\C,\\
						(\theta,x)&\longmapsto
						r^{-(\kappa-m)}
						F_r(\theta,\theta+\alpha+r^{\kappa-m}x).
					\end{aligned}
					\]
				Substituting $h=r^{\kappa-m}x$ in
					\Cref{eq:F_r_estimate} gives the limiting equation
					$G_\zeta = 0$, where
					$c_\zeta=\overline{a_\kappa}(1-\zeta^{-\kappa})\ne0$ and
					\[
					\begin{aligned}
						G_\zeta:(\R/2\pi\Z)\times\R&\longrightarrow\C,\\
						(\theta,x)&\longmapsto
						-imxe^{im\theta}+c_\zeta e^{-i\kappa\theta}.
					\end{aligned}
					\]
					At any zero of $G_\zeta$,
					\[
						x=\frac{c_\zeta}{im}e^{-i(m+\kappa)\theta},
					\]
					so all zeros of $G_\zeta$ have
					$|x|=|c_\zeta|/m$.

						\smallskip
						\noindent\emph{Step 2: small perturbation.}
				By \Cref{eq:sector-displacement-bound}, the corresponding
					 values of $x$ at the zeros of $F_r$ in
				$\mathcal S_j(\varepsilon_{\mathrm{sec}})$ remain bounded as 
				$r\to0$.  Choose $R>0$ so that, for all $r$ small enough, these
					zeros have $|x|<R$. And also so that all zeros of
					$G_\zeta$ also have $|x|<R$.  Thus the zeros of
						$F_{r, \zeta}$ corresponding, under
						$(\theta,x)\mapsto
						(\theta,\theta+\alpha+r^{\kappa-m}x)$, to zeros of $F_r$ in
						$\mathcal S_j(\varepsilon_{\mathrm{sec}})$, and the zeros of
						$G_\zeta$ all lie in the compact subset
					\[
						(\R/2\pi\Z)\times[-R,R],
					\]
					and none lies on its boundary.  Since $R$ is now
					fixed, after decreasing $r_0$ if necessary we may assume
					\[
						r^{\kappa-m}R<\varepsilon_{\mathrm{sec}}
					\]
				for every $0<r<r_0$.  Then, for $|x|\le R$, the point 
				$(\theta,\eta)
				=
				(\theta,\theta+\alpha+r^{\kappa-m}x)
				$
				lies in $\mathcal S_j(\varepsilon_{\mathrm{sec}})
				\subset(\R/2\pi\Z)^2$. On the compact cylinder
				$(\R/2\pi\Z)\times[-R,R]$, putting
				$h=r^{\kappa-m}x$ in \Cref{eq:F_r_estimate} yields that
					\[
						F_{r,\zeta}\longrightarrow G_\zeta
					\]
					uniformly. And the same is true after differentiating with
					respect to $\theta$ and $x$.  Thus
					$F_{r,\zeta}-G_\zeta$, together with its first derivatives,
						is uniformly small on the whole compact cylinder.

					\smallskip
					\noindent\emph{Step 3: zeros of the limiting equation.}
					We now count the zeros of $G_\zeta$.  Since $x \in \R$,
					the equation displayed in Step 1 is equivalent to
					\[
						\arg\left(\frac{c_\zeta}{im}\right)
						-(m+\kappa)\theta\in\pi\Z.
					\]
					As $\theta$ runs over $\R/2\pi\Z$, this equation has
					exactly $2(m+\kappa)$ solutions.  At each of them
					$x\ne0$, because $c_\zeta\ne0$. These zeros are 
					nondegenerate.  Indeed,
					\[
						\partial_xG_\zeta=-ime^{im\theta},
					\]
					and, at a zero of $G_\zeta$ we have
					\[
					\begin{aligned}
						\partial_\theta G_\zeta
						&=m^2xe^{im\theta}-i\kappa c_\zeta e^{-i\kappa\theta}\\
						&=m^2xe^{im\theta}+\kappa mxe^{im\theta}
						=m(m+\kappa)xe^{im\theta}.
					\end{aligned}
					\]
					Thus $\partial_xG_\zeta$ is a real multiple of
					$ie^{im\theta}$, and $\partial_\theta G_\zeta$ is a
					 real multiple of $e^{im\theta}$.  They are
					real linearly independent.

					\smallskip
					\noindent\emph{Step 4: stability and transversality of the 
					intersections.}
					Take disjoint small neighborhoods of the zeros of $G_\zeta$.
					On the complement, $|G_\zeta|$ has a positive minimum.
						On one hand, since $F_{r,\zeta}$, and also its first 
						derivatives,
						converge uniformly to $G_\zeta$ on the compact subset, no
						zeros of $F_{r,\zeta}$ occur on this complement for
						$r\ll1$. On another hand, the implicit function theorem 
						gives one and
						only one zero of $F_{r, \zeta}$ near each zero of
						$G_\zeta$ (equivalently, we are just using that 
						nondegenerate zeros are stable under $C^1$ 
						perturbations).  Moreover, the full real rank of
						$D G_\zeta$ at each limiting zero persists at the
						corresponding zero of $F_{r,\zeta}$.  Hence, for every
						sufficiently small $r>0$,
				the sector $\mathcal S_j(\varepsilon_{\mathrm{sec}})$ contains exactly
					$2(m+\kappa(\zeta))$ ordered solutions of
					\[
						z_r(\theta)=z_r(\eta),
						\quad
						\theta\ne\eta,
					\]
							and $D F_{r,\zeta}$ has full real rank at each corresponding
							zero in the variables $(\theta,x)$.
					Since the change of variables from $(\theta,x)$ to
					$(\theta, \eta)$ has nonzero Jacobian,
					full real rank of $D F_{r,\zeta}$ at a zero is equivalent
					to full real rank of
					\[
						D_{\theta,\eta}\bigl(z_r(\theta)-z_r(\eta)\bigr).
					\]
					This is equivalent to the real linear independence of
					$z_r'(\theta)$ and $z_r'(\eta)$, hence to distinct
					tangent lines at the point $z_r(\theta)=z_r(\eta)$.
					Thus all the self-intersections found in this sector are
					transverse.

				\smallskip
				\noindent\emph{Putting everything together.}
				Now we just sum over all sectors and get
				\[
					2\sum_{j=1}^{m-1}
					\bigl(m+\kappa(e^{2\pi ij/m})\bigr)
					=
					2\sum_{\substack{\zeta^m=1\\ \zeta\ne1}}
					\bigl(m+\kappa(\zeta)\bigr)
					\]
					ordered pairs $(\theta,\eta)$ with $\theta \ne \eta$ such 
					that
					$z_r(\theta)=z_r(\eta)$.  Here we are counting each 
					unordered pair twice. Therefore (recall 
					\Cref{def:pairwise-double-count})
				\[
					\dbl\bigl(z_r(\R/2\pi\Z)\bigr)
					=
					\sum_{\substack{\zeta^m=1\\ \zeta\ne1}}
					\bigl(m+\kappa(\zeta)\bigr).
				\]
				Using \Cref{lem:root-of-unity-separation-sum}, we get
				\[
	\dbl\bigl(z_r(\R/2\pi\Z)\bigr)
					=m(m-1)+2\delta(Q,0)+m-1
					=2\delta(Q,0)+m^2-1.
				\]
					The transversality
					of Step 4  shows that every multiple point is an
					ordinary multiple point (recall \Cref{def:predivide}).  
			\end{proof}
		The proof of the main theorem of this section now follows easily:
		\begin{proof}[Proof of
			\Cref{thm:distinct-tangent-conjugate-pair-predivide}]
				By \Cref{lem:distinct-basic-estimates}, the maps
				$\alphacp{r}$ form a real analytic family in $r$ and are immersions for
				$0<r\ll1$.  By
			\Cref{prop:direct-root-of-unity-double-count}, the image
			$z_r(\R/2\pi\Z)$ has only ordinary multiple points and
			\[
			\dbl\bigl(z_r(\R/2\pi\Z)\bigr)
			=
			2\delta(Q,0)+m^2-1.
				\]
					By \Cref{rem:linear-reduction-to-zr}, the same statements and the
					same value of $\dbl$ hold for
				\[
				\Gamma^{\mathrm{cp}}_{\boldsymbol\gamma,r}
				=
				\alphacp{r}(\R/2\pi\Z).
				\]
			By \Cref{eq:conjugate-alpha-uniform-smallness}, choose a closed real
			disk $D\subset\mathbb B_\R$, with
			$0 \in \operatorname{int}D$, and
			decrease $r_0$ so that this curve lies in $\operatorname{int}D$ for
			all $0<r<r_0$.  Hence
			$\Gamma^{\mathrm{cp}}_{\boldsymbol\gamma, r}$ is a predivide in 
			$D$. Since $Q$ and $\overline Q$ have distinct tangent lines,
			$
			(Q\cdot\overline Q)_0=m^2
			$.  Therefore
			\[
			\dbl(\Gamma^{\mathrm{cp}}_{\boldsymbol\gamma,r})
			=
			2\delta(Q,0)+(Q\cdot\overline Q)_0-1.
			\]
			Using \Cref{eq:delta-branch-formula} and
			$\delta(\overline Q,0)=\delta(Q,0)$, we get
			\[
			\dbl(\Gamma^{\mathrm{cp}}_{\boldsymbol\gamma,r})
			=
			\delta(Q\cup\overline Q,0)-1.
			\]
			\end{proof}

		\subsection*{Examples and limitations of the trace map}

		The first example treats the singularity singled out by Leviant and 
		Shustin as the first case not covered by their results.
		The other two examples show that our construction does not work if the 
		pair of complex conjugate branches has the same tangent or 
		or if we don't choose a Newton--Puiseux parametrization (more 
		concretely, if $p\circ\gamma_Q(u)\ne u^m$).

		\begin{example}
		\label{ex:leviant-shustin-first-missing-example}
		Leviant and Shustin identify the following singularity as the simplest
		case beyond the range of their Theorem 1
		\cite[Example 1(3)]{LeviantShustin2018Morsifications}.  In real
		coordinates $(x,y)$, put
		\[
			w_\pm=y\pm ix.
		\]
		The singularity is defined by
		\[
			\left((w_+^2-x^3)^2-x^5w_+\right)
			\left((w_-^2-x^3)^2-x^5w_-\right)=0.
		\]
		Let $Q$ be the branch defined by the first factor.  Together with its
		conjugate branch, it forms an orbit consisting of a conjugate pair in the sense of
		\Cref{sec:trace-map}.  Write
		\[
			x=t^4,
			\qquad
			w_+=t^6\lambda(t),
		\]
		where $\lambda(t)$ is the analytic solution of 
		$(\lambda(t)^2-1)^2=t^2\lambda(t)$
		satisfying $\lambda(0)=1$ and $\lambda'(0)=1/2$.
		Its expansion in power series starts by
		\[
			\lambda(t)
			=
			1+\frac12t-\frac1{64}t^3+\frac1{128}t^4-\frac9{4096}t^5+O(t^6).
		\]
		Thus $m=4$ and $(\beta_1,\beta_2)=(6,7)$. The tangent line of $Q$
		is $\{w_+=0\}$.  Set
		$q=-iw_+/2$, so $\ker q = T_0Q$, and define $p$ from $q$ by
		\Cref{eq:distinct-conjugate-adapted-p-coordinate}. So we have
		\[
			p=\frac{x+iy}{2}=x+\frac i2w_+,
			\qquad
			q=\frac{x-iy}{2}=-\frac i2w_+.
		\]
		With respect to these coordinates, the real structure has the expression
		in \Cref{eq:distinct-conjugate-real-structure}, and the tangent lines of
		$Q$ and $\overline Q$ are distinct.  We find a
		Newton--Puiseux parametrization as in
		\Cref{eq:distinct-conjugate-np-parametrization} and get
		$
			\gamma_Q(u)=(u^4,\varphi(u)),
		$
		with
		\begin{equation}\label{eq:leviant-shustin-normalized-puiseux}
		\begin{aligned}
			\varphi(u)
			&=
			-\frac i2u^6-\frac i4u^7-\frac38u^8
			+\left(-\frac{13}{32}+\frac i{128}\right)u^9
			\\
			&\quad
			+\left(-\frac7{64}+\frac{83i}{256}\right)u^{10}
			+O(u^{11}).
		\end{aligned}
		\end{equation}
		The conjugate branch is parametrized, as in
		\Cref{eq:distinct-conjugate-conjugate-parametrization}, by
		$
			\gamma_{\overline Q}(v)=(\overline{\varphi}(v),v^4).
		$
		Set
		\[
			\boldsymbol\gamma=(\gamma_Q,\gamma_{\overline Q}).
		\]
		For drawing matters of
		\Cref{fig:leviant-shustin-first-missing-divide}, we discard the
		terms of order $O(u^{11})$ in
		\Cref{eq:leviant-shustin-normalized-puiseux}.  This leaves
		$m,\beta_1,\beta_2$, and hence the values of $\kappa$, unchanged by
		\Cref{lem:root-of-unity-separation-sum}.  Therefore
		\Cref{prop:direct-root-of-unity-double-count} gives the same value of
		$\dbl$.
		Equivalently,  the drawn curve is the image in $\C$ of
		\[
		\begin{aligned}
			z_r^{\le10}(\theta)
			&=
			e^{4i\theta}
			+\frac i2r^2e^{-6i\theta}
			+\frac i4r^3e^{-7i\theta}
			-\frac38r^4e^{-8i\theta}
			\\
			&\quad
			+\left(-\frac{13}{32}-\frac i{128}\right)r^5e^{-9i\theta}
			+\left(-\frac7{64}-\frac{83i}{256}\right)r^6e^{-10i\theta}.
		\end{aligned}
		\]
		The sequence of greatest common divisors defined in
		\Cref{eq:characteristic-gcd-sequence} is
			\[
				d_0=4,
				\qquad
				d_1=\gcd(4,6)=2,
				\qquad
				d_2=\gcd(4,6,7)=1.
			\]
			By \Cref{lem:root-of-unity-separation-sum}, the roots
			$i$ and $-i$ have $\kappa=6$, and the root $-1$ has
			$\kappa = 7$.  Hence
			\Cref{prop:direct-root-of-unity-double-count} (recall also 
			\Cref{rem:linear-reduction-to-zr}), gives
			\[
			\dbl(\Gamma^{\mathrm{cp}}_{\boldsymbol\gamma,r})
			=(4+6)+(4+6)+(4+7)
			=31.
			\]
		Here $\delta(Q,0)=8$.  Thus $31$ agrees with the
		formula in \Cref{thm:distinct-tangent-conjugate-pair-predivide}:
		\[
			2\delta(Q,0)+(Q\cdot\overline Q)_0-1
			=
			2\cdot 8+4^2-1
			=
			31.
		\]

		\begin{figure}[!ht]
			\centering
			\tikzsetnextfilename{first-missing-divide}
			\scalebox{0.9}{\input{figures/first_missing_divide.tikz}}
			\caption{The predivide for the Leviant--Shustin example.  The displayed
			curve is the image of $z_r^{\le10}$.  By
			\Cref{rem:linear-reduction-to-zr}, $L_r$ preserves the pairs of
			parameters with the same image, the transversality of the corresponding
			intersections, and $\dbl$.}
			\label{fig:leviant-shustin-first-missing-divide}
		\end{figure}
	\end{example}

	\begin{example}
		\label{ex:same-tangent-not-predivide}
		Use complex coordinates $(x, y)$ for which the real structure is
		\[
			(x,y)\longmapsto(\overline x,\overline y).
		\]
		Let $Q$ be the branch parametrized by $\gamma_Q(u)=\left(u^2,u^3+i 
		u^4\right)$. The conjugate branch is parametrized by $\gamma_{\overline 
		Q}(v)=\left(v^2,v^3-i v^4\right)$.
		Set $\boldsymbol\gamma=(\gamma_Q,\gamma_{\overline Q})$.
		The two tangent lines are the same real line $\{y=0\}$.  The two
		branches are distinct.  Indeed, if they were equal, then, because the
		first coordinate is $u^2$, there would be a root of unity
		$\xi^2=1$ such that
		\[
			u^3+i u^4
			=
			(\xi u)^3-i(\xi u)^4 .
		\]
		If $\xi=1$, comparison of the coefficients of $u^4$ gives
		$i=-i$.  If $\xi=-1$, comparison of the coefficients of $u^3$
		gives $1 = -1$.  Both cases are impossible.

		For $s=r^2>0$, the circle map
		$\alphacp{r}$ of
		\Cref{eq:conjugate-pair-circle-map} is
		$
			\alphacp{r}(\theta)=
			\bigl(
			2r^2\cos(2\theta),\,
			2r^3\cos(3\theta)-2r^4\sin(4\theta)
			\bigr).
		$
		At $\theta_0=\frac{\pi}{2}$ and $\eta_0=\frac{3\pi}{2}$, one has
		\[
			\alphacp{r}(\theta_0)
			=
			\alphacp{r}(\eta_0)
			=
			(-2r^2,0).
		\]
		The derivatives at these two parameter values are
		\[
			\bigl(\alphacp{r}\bigr)'(\theta_0)
			=
			(0,\,6r^3-8r^4),
			\text{ and }
			\bigl(\alphacp{r}\bigr)'(\eta_0)
			=
			(0,\,-6r^3-8r^4).
		\]
		For every $r>0$  small enough, both vectors are nonzero and 
		generate the same line.  So 
		$\Gamma^{\mathrm{cp}}_{\boldsymbol\gamma,r}$ 
		has a
		self-intersection point which is actually not a node, but a {\em 
		self-tangency}.  By
		\Cref{def:predivide}, this
		$\Gamma^{\mathrm{cp}}_{\boldsymbol\gamma, r}$ is not a predivide.
		Thus the conclusion of
		\Cref{thm:distinct-tangent-conjugate-pair-predivide} does not always hold
		when the two branches have the same tangent line.

		\begin{figure}[!ht]
			\centering
			\includegraphics[width=.4\linewidth]
			{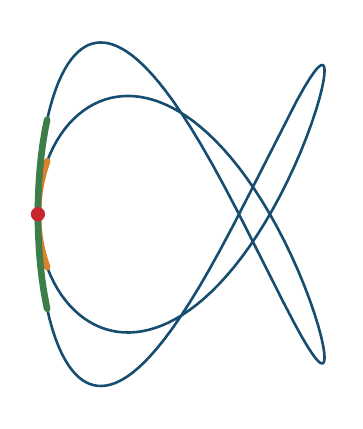}
			\caption{The curve $\Gamma^{\mathrm{cp}}_{\boldsymbol\gamma,r}$ in
			\Cref{ex:same-tangent-not-predivide}.  The marked arcs meet at the red
			point with the same tangent line.}
			\label{fig:same-tangent-failure}
		\end{figure}

	\end{example}

	\begin{example}
		\label{ex:adapted-source-coordinate-needed}
		Use complex coordinates $(p, q)$ in which the real structure has the
		expression $(p,q)\mapsto(\overline q,\overline p)$ (recall 
		\Cref{eq:distinct-conjugate-real-structure}). 
		Let $Q$ be the branch parametrized by $
			\widetilde\gamma_Q(w)=(w^2,w^3)$. 
		The induced conjugate parametrization is $\widetilde\gamma_{\overline 
		Q}(z)=(z^3,z^2)$. 
		These parametrizations have precisely the form in
		\Cref{eq:distinct-conjugate-np-parametrization,eq:distinct-conjugate-conjugate-parametrization},
		with $m = 2$ and $\varphi(w)=w^3$.
		Set
		\[
			\widetilde{\boldsymbol\gamma}
			=
			(\widetilde\gamma_Q,\widetilde\gamma_{\overline Q}).
		\]
		The tangent line of $Q$ is $\{q=0\}$, and the conjugate branch has 
		tangent
		line $\{p=0\}$.  That is, the conjugate tangent lines are distinct.  
		Here
		$m=2$, and $\delta(Q,0)=1$ by
		\Cref{lem:root-of-unity-separation-sum}.  Since the tangent lines are
		distinct, $(Q\cdot\overline Q)_0=m^2=4$.  For the parameter $w$, the value of
		$\dbl(\Gamma^{\mathrm{cp}}_{\widetilde{\boldsymbol\gamma},r})$ predicted by
		\Cref{thm:distinct-tangent-conjugate-pair-predivide} is therefore
		\[
			2\delta(Q,0)+(Q\cdot\overline Q)_0-1=5.
		\]
		Next, we follow the construction using a different primitive 
		parametrization. We consider the parametrization 
		\[
			\gamma_Q(u)=\bigl((u+u^2)^2,(u+u^2)^3\bigr)
			\] 
		and, its {\em conjugate},
		\[
			\gamma_{\overline Q}(v)
			=
			\bigl((v+v^2)^3,(v+v^2)^2\bigr).
		\]
	which parametrize the same pair of complex 
		conjugate branches.
		On the circle $\widehat{\mathcal R}_{s}$ from
		\Cref{eq:positive-real-source-circle} (with $s=r^2$) put 
			$u=re^{i\theta}, v=re^{-i\theta}$ and $
			w_\theta=re^{i\theta}+r^2e^{2i\theta}$.
		Let $\boldsymbol\gamma=(\gamma_Q,\gamma_{\overline Q})$.  By
		\Cref{eq:conjugate-pair-circle-map}, the first coordinate of
		$\alphacp{r}$ is
		\[
			f_r(\theta)=w_\theta^2+\overline{w_\theta}^{\,3},
		\]
		and the second coordinate is $\overline{f_r(\theta)}$.  Hence pairs of
		distinct parameters with the same image are exactly the zeros of
		$f_r(\theta)-f_r(\eta)$ with $\theta\ne\eta$. Expanding $f_r$ and 
		applying the argument used to localize the solutions in the proof of
		\Cref{prop:direct-root-of-unity-double-count} shows that $\alphacp{r}$ is
		an immersion for $0<r\ll1$ and that every off-diagonal solution of
		$f_r(\theta)=f_r(\eta)$ can be written as
		\[
			\eta=\theta+\pi+rx,
		\]
		with $x$ in a fixed compact set.  In these variables, a direct
		calculation shows that
		\[
			r^{-3}e^{-2i\theta}
			\bigl(f_r(\theta)-f_r(\theta+\pi+rx)\bigr)
			\longrightarrow
			-2ix+4e^{i\theta}+2e^{-5i\theta}.
		\]
		The limiting equation has exactly two zeros:
		\[
			(\theta,x)=(\pi/2,1),
			\text{ and }
			(\theta,x)=(3\pi/2,-1).
		\]
		At both zeros, $\partial_x=-2i$ and $\partial_\theta=\mp14$.  Hence
		both zeros are nondegenerate.
		The argument for the stability of the zeros in
		\Cref{prop:direct-root-of-unity-double-count} therefore gives two ordered
		solutions for  $r>0$ small enough.  They represent one
		unordered pair, as shown in
		\Cref{fig:adapted-source-coordinate-needed}.  Hence
		\[
			\dbl(\Gamma^{\mathrm{cp}}_{\boldsymbol\gamma,r})=1.
		\]

		\begin{figure}[htbp]
			\centering
			\includegraphics[width=.65\linewidth]
			{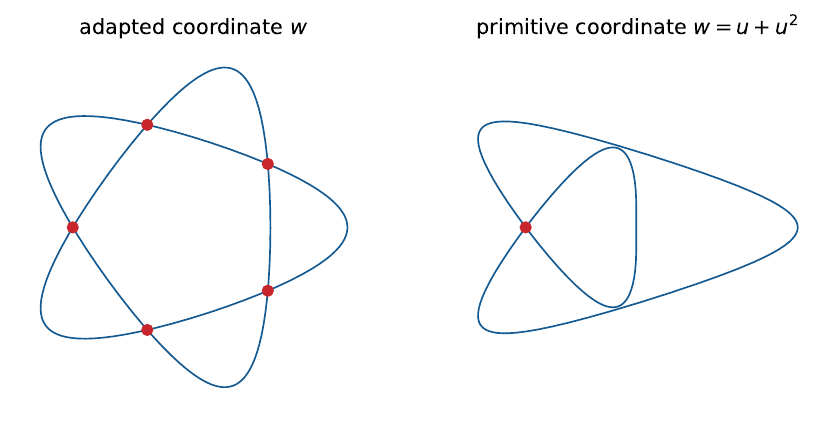}
			\caption{The curves $\Gamma^{\mathrm{cp}}_{\widetilde{\boldsymbol\gamma},r}$
			and $\Gamma^{\mathrm{cp}}_{\boldsymbol\gamma,r}$ for the two 
			parametrizations.}
			\label{fig:adapted-source-coordinate-needed}
		\end{figure}
		This example shows that an arbitrary parametrization does not
		necessarily give the correct number of double points and so 
		Newton--Puiseux parametrizations are essential in our theorems.
	\end{example}

	\section{Construction for all orbits of branches}
	\label{sec:global-predivide}

	The proof follows a similar technique as Leviant
	and Shustin in their proof of Theorem~1
	\cite[Sections~2.1.1 and~2.3]{LeviantShustin2018Morsifications} which, in 
	turn, is a generalization of the induction by translations and contractions 
	by A'Campo in \cite{ACampo1975MonodromyDeploiement}. Our
	\Cref{thm:distinct-tangent-conjugate-pair-predivide} provides parametrizations
	for the orbits of conjugate branches not covered in the Leviant and Shustin work 
	and that's the extra input here.  We combine these
	parametrizations of these circles with the
	parametrizations of the real branches, perturb them before each contraction,
	and count their intersections throughout the descent.

	Leviant and Shustin note, following S.~M.~Gusein-Zade, that the function
	proposed in Formula~(1) on p.~12 of A'Campo's paper does not have the
	claimed properties
	\cite[Theorem~1, \S~2]{ACampo1975MonodromyDeploiement}
	\cite[p.~308, footnote~2]{LeviantShustin2018Morsifications}.  We do not use
	this formula.  Instead, we carry out the descent using parametrizations.
	 Once the parametrizations and marked
	intersections at one stage have been fixed, we choose a higher power of
	$\lambda$ for the next divide or translation.  The circle then lies at a
	smaller scale, while the translation is chosen to  preserve the earlier 
	transverse intersections (see
	step~\ref{it:global-algorithm-translations}).  Then, after the descent is
	finished, Remmert's result says that the (deformations) of the 
	parametrizations can be given by an equation. In other words, we put 
	together
	the parametrizations
	into the map in \Cref{eq:global-normalization-map}, restrict it to a proper
	finite map, and take its analytic image. 

	Next, we give a roadmap of the proof in order to make it easier to follow. 
	It also help us introduce some notation.
		\begin{roadmap}\phantomsection\label{roadmap:global-construction}
		\begin{enumerate}
			\item\label{it:global-algorithm-resolution}
			\emph{Resolution.}
			We blow up the real points where strict transforms still meet or where the
			total transform fails to have normal crossings.  We continue until
			the real strict transforms are smooth and transverse to the exceptional
			divisor and also, until every conjugate pair of branches has 
			separated after its last common real
			infinitely near point.  The resulting modification of $X_0$ is 
			denoted by $\rho:X_N=X_\rho\longrightarrow X_0$.

			\item\label{it:global-algorithm-initialization}
			\emph{Real branches.}
			For every strict transform of a real branch, we choose a primitive 
			parametrization
			defined on a normalization disk.  This disk contains an interval 
			that is fixed by the
			real structure. This interval parametrizes the real arc used in the 
			descent, that 
			will eventually yield the divide for this branch.

			\item\label{it:global-algorithm-packets}
			\emph{Packets.}
			At each step of the descent, we consider every point $c$ on the current
			surface that is the last common (real) infinitely near point of at
			least one
			conjugate pair.  The conjugate pairs associated with $c$ form the packet
			$\mathcal P_c$.  We consider a partition of this {\em packet} into 
			the
			unordered pairs
			of conjugate tangent lines.  For every $\lambda>0$ small enough, we 
			apply
			\Cref{thm:distinct-tangent-pair-families} to these pairs. The 
			circle maps obtained in this way,
			together with the real arcs already passing through $c$, 
			parametrize a local predivide at $c$.

			\item\label{it:global-algorithm-translations}
			\emph{Translations.}
			Before the next contraction, we translate the parametrized predivide in
			real local coordinates, using a new (high) power of $\lambda$.  We 
			choose the
			translation so that the predivide meets the remaining exceptional
			components in pairwise distinct transverse real points away from their
			crossings.  

			\item\label{it:global-algorithm-first-contraction}%
			\label{it:global-algorithm-assembly}
			\emph{Contractions.}
			We contract the next exceptional components in reverse order.  Distinct
			transverse intersections with a contracted component give smooth germs with
			distinct tangent lines at the image of this component.  On the resulting
			surface, we construct the local predivide for every packet as in
			step~\ref{it:global-algorithm-packets} and translate the parametrizations as
			in step~\ref{it:global-algorithm-translations}.  We repeat this process until
			we reach $X_0$.

			\item\label{it:global-algorithm-count}%
			\label{it:global-algorithm-equations}
			\emph{Counting the nodes and finding an equation.}
			We record the unordered pairs of distinct points of the normalization that acquire
			the same image during the construction.  For each new circle or translation,
			we use a power of $\lambda$ higher than the preceding ones.
			After the final contraction, small perturbations of the local parametrizations
			of the branches replace each ordinary point of multiplicity $k$ by
			$\binom{k}{2}$ nodes.  The formulas that control the changes of
			$\delta$ and
			intersection multiplicity after a blowup give
			$\delta(C, 0)-\imbr(C, 0)$ nodes.
		\end{enumerate}
	\end{roadmap}

	\subsection*{Circle maps at a real infinitely near point}
	We define the notion of {\em adapted local data} which will be repeatedly 
	use in this section.
 	Let $X$ be a smooth complex surface with real structure $\sigma_X$, let
	$c\in X_\R$, and let $\mathcal T$ be a finite set of nonreal complex
	lines in $T_cX$.  Let $\mathcal P$ be a finite set of conjugate pairs
	$\mathfrak q = \{Q, \overline Q\}$, labeled so that
	$T_{\mathfrak q}:=T_cQ$ belongs to $\mathcal T$.
	The {\bf adapted local data} for
	$(X,c,\mathcal T,\mathcal P)$ consists of a
	coordinate
	neighborhood $(U, \vartheta)$ with $U$ invariant by $\sigma_X$, one complex 
	linear coordinate change $H_T$
	for each $T\in\mathcal T$, and one compatible pair of primitive
	parametrizations $\boldsymbol\gamma_{\mathfrak q}$ for each
	$\mathfrak q\in\mathcal P$.  These choices are specified as follows.
	 We can choose a neighborhood $U$
	of $c$, invariant under $\sigma_X$, and a
	centered holomorphic chart
	\[
		\vartheta:(U,c)\longrightarrow(\C^2,0)
	\]
	in which $\sigma_X$ is coordinatewise conjugation (recall 
	\Cref{rem:local-coordinates-real-structure}).  Denote
	$U_\R=U\cap X_\R$ the corresponding real part.  For each $T\in\mathcal T$, 
	choose a complex linear
	isomorphism $H_T:\C^2\to\C^2$, and put
	$\chi_T=H_T\circ\vartheta$.  Writing $(\xi, \eta)$ for the standard
	coordinates on the target $\C^2$, the adapted local data allows us to 
	choose $H_T$ so that 
	\[
		d(\chi_T)_c(T)=\C\times\{0\}
		\quad \text{ and } \quad
		\chi_T\circ\sigma_X\circ\chi_T^{-1}(\xi,\eta)
		=(\overline\eta,\overline\xi),
	\]
	hold.
	These isomorphisms exist because $T\ne\overline T$ for every
	$T\in\mathcal T$.  For $\mathfrak q=\{Q,\overline Q\}\in\mathcal P$, put
	$m_{\mathfrak q}=\mt(Q,c)$ and, as in
	\Cref{eq:distinct-conjugate-np-parametrization,eq:distinct-conjugate-conjugate-parametrization},
	choose primitive parametrizations
	$\boldsymbol\gamma_{\mathfrak q}
	=(\gamma_Q,\gamma_{\overline Q})$, with coordinates $u$ and $v$ on the
	respective source disks, such that
	\[
	\begin{aligned}
		\chi_{T_{\mathfrak q}}\circ\gamma_Q(u)
		&=(u^{m_{\mathfrak q}},\varphi_{\mathfrak q}(u)),\\
		\chi_{T_{\mathfrak q}}\circ\gamma_{\overline Q}(v)
		&=(\overline{\varphi_{\mathfrak q}}(v),v^{m_{\mathfrak q}}),
	\end{aligned}
	\qquad
	\varphi_{\mathfrak q}(u)
	=\sum_{n>m_{\mathfrak q}}a_{\mathfrak q,n}u^n.
	\]
	Here $\overline{\varphi_{\mathfrak q}}(v)
	=\sum_{n>m_{\mathfrak q}}\overline{a_{\mathfrak q,n}}v^n$.
	For $r>0$ small, the addition in
	\Cref{eq:conjugate-pair-circle-map} is performed in the chart $\vartheta$:
	\[
		\alpha^{\mathrm{cp}}_{\boldsymbol\gamma_{\mathfrak q},r}(\theta)
		=
		\vartheta^{-1}\!\left(
			\vartheta\circ\gamma_Q(re^{i\theta})
			+
			\vartheta\circ\gamma_{\overline Q}(re^{-i\theta})
		\right).
	\]
	The next two lemmas and their proofs use adapted local data in this sense.

	\begin{lemma}
		\label{thm:common-tangent-annuli}
		Let $X$ be a smooth complex surface with real structure $\sigma_X$ and 
		let
		$c\in X_\R$ be a point in its real part.  Let
		$\mathcal P=\{\mathfrak q_1,\ldots,\mathfrak q_N\}$ be a finite
		set of conjugate pairs of branches at $c$.
		Suppose that there is a complex line
		$L$ through $c$, with $L\ne\overline L$, such that for every pair 
		$\mathfrak q_i=\{Q_i,\overline Q_i\}$ we have
		$T_cQ_i = L$.
		Fix adapted local data for
		$(X,c,\{L\},\mathcal P)$.
		Write $m_i=m_{\mathfrak q_i}$ and
		$\boldsymbol\gamma_i=\boldsymbol\gamma_{\mathfrak q_i}$. Let $M$ be a 
		positive integer divisible by $m_1,\ldots,m_N$ and set
		$\rho_i(\lambda)=\lambda^{M/m_i}$. Let $W$ be a small open neighborhood 
		of $c$ in $X_\R$. Then there exist
		$\lambda_0>0$ and a closed real disk $D\subset W\cap U_\R$, with
		$c \in D$, such that, for every
		$0<\lambda<\lambda_0$, the union of the maps
		\[
			\alpha^{\mathrm{cp}}_{\boldsymbol\gamma_i,\rho_i(\lambda)}
			:
			\R/2\pi\Z\longrightarrow D,
			\qquad i=1,\ldots,N,
		\]
		parametrize a predivide in $D$.  If we denote by
		$A_{\mathfrak q_i,\lambda}
		=\alpha^{\mathrm{cp}}_{\boldsymbol\gamma_i,\rho_i(\lambda)}
		(\R/2\pi\Z)\subset\operatorname{int}D$,  then
		\[
			\dbl(A_{\mathfrak q_i,\lambda})
			=
			\delta(Q_i\cup\overline Q_i,c)-1,
			\qquad i=1,\ldots,N,
		\]
		and
		\[
			\dbl(A_{\mathfrak q_i,\lambda},
			A_{\mathfrak q_j,\lambda})
			=
			\sum_{A\in\mathfrak q_i,\,B\in\mathfrak q_j}
			(A\cdot B)_c
			=
			2(Q_i\cdot Q_j)_c+2m_im_j,
			\qquad 1\le i<j\le N.
		\]
		Put $A_{\mathcal P,\lambda}
		=\bigcup_{i=1}^NA_{\mathfrak q_i,\lambda}$ and
		$C_{\mathcal P}=\bigcup_{i=1}^N(Q_i\cup\overline Q_i)$.  Then
		\[
			\dbl(A_{\mathcal P,\lambda})
			=
			\delta(C_{\mathcal P},c)-N.
		\]
	\end{lemma}

	\begin{proof}
		Fix a neighborhood $W$ of $c$ in $X_\R$.  We use the adapted local
		data fixed in the statement. That is, we work in the coordinates 
		$\chi_L$ and write
		$\varphi_i=\varphi_{\mathfrak q_i}$ and
		$a_{i,n}=a_{\mathfrak q_i,n}$.
		Since $H_L$ is complex linear and
		$\rho_i(\lambda)^{m_i}=\lambda^M$, using the radius
		$\rho_i(\lambda)$ in
		\Cref{eq:conjugate-pair-circle-map} gives
		\[
			\chi_L(A_{\mathfrak q_i,\lambda})
			=
			\bigl\{
			\lambda^M
			\bigl(z_{i,\lambda}(\theta),
			\overline{z_{i,\lambda}(\theta)}\bigr):
			\theta\in\R/2\pi\Z
			\bigr\},
		\]
		where
		\begin{equation}
			\label{eq:common-tangent-normalized-circle}
			z_{i,\lambda}(\theta)
			=
			e^{im_i\theta}
			+
			\sum_{n>m_i}
			\overline{a_{i,n}}\,
			\rho_i(\lambda)^{\,n-m_i}e^{-in\theta}.
		\end{equation}
		Since $L \ne \overline L$,
		\Cref{thm:distinct-tangent-conjugate-pair-predivide} shows that every
		circle map is an immersion, that every multiple point on one circle is
		ordinary, and that
		\[
			\dbl(A_{\mathfrak q_i,\lambda})
			=
			\delta(Q_i\cup\overline Q_i,c)-1.
		\]

		Fix $i < j$.  Since $\lambda>0$ and
		$z\mapsto(z,\overline z)$ is injective, the points with parameters $\theta$
		and $\psi$ have the same image precisely when $
			z_{i,\lambda}(\theta)-z_{j,\lambda}(\psi)=0$.
		Put
		\[
			d=\gcd(m_i,m_j),
			\qquad
			\ell=\operatorname{lcm}(m_i,m_j)
			=\frac{m_im_j}{d}.
		\]
		Since $M$ is divisible by $m_i$ and $m_j$, it is divisible by
		$\ell$, so $M/\ell$ is an integer.

		The equation $u^{m_i}=w^{m_j}$, which equates the first coordinates of
		$Q_i$ and $Q_j$, defines a plane curve germ in the
		space with coordinates $(u,w)$.  It has $d$ irreducible components,
		indexed by $\nu = 1, \ldots, d$.  For each $\nu$, choose $\zeta_\nu$ with
		$\zeta_\nu^{m_j}=1$ such that the normalization of the component indexed by $\nu$ is
		parametrized by
		\begin{equation}
			\label{eq:common-tangent-component-parametrization}
			u=u_\nu(\tau)=\tau^{\ell/m_i},
			\qquad
			w=w_\nu(\tau)=\zeta_\nu\tau^{\ell/m_j}.
		\end{equation}
		For every $\nu$, the function
		$\varphi_i(u_\nu(\tau))-\varphi_j(w_\nu(\tau))$ is not identically zero.
		Otherwise the two parametrizations obtained from this component would
		parametrize the same germ,
		contrary to $Q_i\ne Q_j$.  Put
		\[
			k_\nu
			=
			\ord_\tau\!
			\left(
				\varphi_i(u_\nu(\tau))
				-
				\varphi_j(w_\nu(\tau))
			\right).
		\]
		The formula for intersection multiplicity in terms of parametrizations
		computes $(Q_i\cdot Q_j)_c$ by
		equating the first coordinates and summing (over the normalized 
		components)
		the orders of the differences of the second coordinates
		\cite[Section~4.1]{Wall2004SingularPoints}.  Therefore
		\begin{equation}
			\label{eq:common-tangent-contact-sum}
			\sum_{\nu=1}^{d}k_\nu=(Q_i\cdot Q_j)_c.
		\end{equation}
		For every $\nu$, one has $k_\nu>\ell$.
		Indeed observe that $\ord_\tau u_\nu=\ell/m_i$ and
		$\ord_\tau w_\nu=\ell/m_j$.  Since all exponents of $\varphi_i$
		are bigger than $m_i$ and all exponents of $\varphi_j$ are bigger than 
		$m_j$, we 
		have
		$k_\nu > \ell$.
		Set
		\[
			q_\nu=\frac{k_\nu}{\ell}-1>0.
		\]
		Then $Mq_\nu = (M/\ell)k_\nu-M$ is a positive integer.  For
		$\tau=\lambda^{M/\ell}e^{it}$, the difference of the terms involving
		$\varphi_i$ and $\varphi_j$ in the two functions
		$z_{i,\lambda}$ and $z_{j, \lambda}$ is
		\[
			\lambda^{-M}
			\overline{
			\varphi_i(u_\nu(\tau))-\varphi_j(w_\nu(\tau))
			}.
		\]
		This expression has order
		$\lambda^{M(k_\nu/\ell-1)}=\lambda^{Mq_\nu}$.

		On the parameter torus, the equation
		$e^{im_i\theta}=e^{im_j\psi}$
		has $d$ connected components, each an embedded circle.  These circles are
		obtained from the normalized components of $u^{m_i}=w^{m_j}$.  For each $\nu$,
		\[
			\gcd(\ell/m_i,\ell/m_j)=1.
		\]
		Hence substitution of
		\[
			\tau=\lambda^{M/\ell}e^{it},
			\quad\enspace
			t\in\R/2\pi\Z,
		\]
		in \Cref{eq:common-tangent-component-parametrization} parametrizes the circle
		obtained from this component exactly once.  The function
		$z_{i, \lambda}(\theta)-z_{j, \lambda}(\psi)$ converges uniformly on the
		 torus to the function
		$e^{im_i\theta}-e^{im_j\psi}$.  Choose
		disjoint tubular neighborhoods of the $d$ circles.  The former
		difference of exponentials is bounded away from zero on the compact 
		complement of the chosen tubular neighborhoods.  Hence, for
	   $\lambda > 0$ small, the  map 
		$e^{im_i\theta}-e^{im_j\psi}$ has no zero there.

  Define the maps
 $\theta_\nu(t),\psi_\nu(t):\R/2\pi\Z \to \R/2\pi\Z$ by the expressions
 \[
 \begin{aligned}
 	u_\nu(\lambda^{M/\ell}e^{it})
 	&=\rho_i(\lambda)e^{i\theta_\nu(t)},\\
 	w_\nu(\lambda^{M/\ell}e^{it})
 	&=\rho_j(\lambda)e^{i\psi_\nu(t)}.
 \end{aligned}
 \]
		So $
			e^{im_i\theta_\nu(t)}
			=
			e^{im_j\psi_\nu(t)}
			=
			e^{i\ell t}$.
		Consider coordinates $(t,h)$ given by
		\[
			\theta=\theta_\nu(t),
			\qquad
			\psi=\psi_\nu(t)+\frac{h}{m_j}.
		\]
		Then $e^{ih}=e^{i(m_j\psi-m_i\theta)}$,
		and the circle is given by $h=0$.  In these coordinates,
		$u = u_\nu(\tau)$ and $w=w_\nu(\tau)e^{ih/m_j}$.  By
		\Cref{eq:common-tangent-normalized-circle},
		\[
		\begin{aligned}
			z_{i,\lambda}(\theta)-z_{j,\lambda}(\psi)
			={}&e^{i\ell t}(1-e^{ih})\\
			&+\lambda^{-M}
			\overline{
			\varphi_i(u_\nu(\tau))
			-
			\varphi_j(w_\nu(\tau)e^{ih/m_j})
			}.
		\end{aligned}
		\]
		Since
		\[
			e^{i\ell t}(1-e^{ih})=-ih e^{i\ell t}+O(h^2),
		\]
		Taylor expansion in $h$ of the term containing $\varphi_i$ and
		$\varphi_j$, together with the definition of $k_\nu$, gives a nonzero coefficient
		$b_\nu$ and
		some $\epsilon>0$ such that, uniformly in $t$ and for small $h$,
		\begin{equation}
			\label{eq:common-tangent-tubular-expansion}
			\begin{aligned}
			z_{i,\lambda}(\theta)-z_{j,\lambda}(\psi)
			={}&-ih e^{i\ell t}
			+b_\nu\lambda^{Mq_\nu}e^{-ik_\nu t}\\
			&+O(h^2)+O(\lambda^\epsilon|h|)
			+O(\lambda^{Mq_\nu+M/\ell}).
			\end{aligned}
		\end{equation}
		The term
		$b_\nu\lambda^{Mq_\nu}e^{-ik_\nu t}$ above is the conjugate of the first
		nonzero term of
		$\varphi_i(u_\nu(\tau))-\varphi_j(w_\nu(\tau))$ (normalized by
		$\lambda^M$).  The error
		$O(\lambda^{Mq_\nu+M/\ell})$ comes from the higher powers of $\tau$.
		Finally, the estimate $O(\lambda^\epsilon|h|)$ above follows because
		every
		exponent of $\varphi_j$ is bigger than $m_j$.  So, after maybe 
		shrinking the
		tubular neighborhoods, and taking
		$\lambda>0$ small enough, we have that 
		\[
			|-ih e^{i\ell t} +O(h^2)+O(\lambda^\epsilon|h|)| \geq |h|/2.
		\]
  Hence, every zero of $z_{i,\lambda}(\theta)-z_{j,\lambda}(\psi)$, satisfies
		\begin{equation}\label{eq:common-tangent-displacement-bound}
			|h|\le C\lambda^{Mq_\nu}.
		\end{equation}

		Put $h=\lambda^{Mq_\nu}x$ and divide
		\Cref{eq:common-tangent-tubular-expansion} by
		$\lambda^{Mq_\nu}$. Since the series that define $z_{i,\lambda}$ and 
		$z_{j,\lambda}$ in
		\Cref{eq:common-tangent-normalized-circle} are convergent, we can 
		differentiate them term by term.  Hence, for every $R>0$, the maps
		\[
			(t,x)\longmapsto
			\lambda^{-Mq_\nu}
			\left(
			z_{i,\lambda}(\theta_\nu(t))
			-
			z_{j,\lambda}\!\left(
			\psi_\nu(t)+\frac{\lambda^{Mq_\nu}x}{m_j}
			\right)
			\right)
		\]
		converge, in $C^1$, on the cylinder
		$(\R/2\pi\Z)\times[-R,R]$ to
		\[
			G_\nu(t,x)
			=
			-ix e^{i\ell t}
			+
			b_\nu e^{-ik_\nu t}.
		\]
		On this cylinder, the three rescaled error terms in
		\Cref{eq:common-tangent-tubular-expansion} are respectively
		$
			O(\lambda^{Mq_\nu}),
			O(\lambda^\epsilon)$ and
			$O(\lambda^{M/\ell})$, 
		with the same bounds after one derivative.  At a zero of $G_\nu$, we 
		have
		\[
			x=-ib_\nu e^{-i(\ell+k_\nu)t}.
		\]
		Since $x$ is real, there are exactly $2(\ell+k_\nu)$ such zeros,
		and each has $|x|=|b_\nu|>0$.  They are nondegenerate because, at each
		zero,
		\[
			\partial_xG_\nu=-ie^{i\ell t},
			\qquad
			\partial_tG_\nu=(\ell+k_\nu)xe^{i\ell t},
		\]
		and these two vectors are real linearly independent.  Choose $R$ larger
		than the bound for $|x|$ obtained from
		\Cref{eq:common-tangent-displacement-bound} and larger than
		$|b_\nu|$.  On the complement of disjoint small neighborhoods of the
		zeros of $G_\nu$ in this compact cylinder, $|G_\nu|$ has a positive
		minimum.  The $C^1$ convergence and the implicit function theorem give
		exactly one zero near each zero of $G_\nu$ and none on the complement.
		By $C^1$ convergence, the derivative of the rescaled difference map has
		full real rank at these zeros whenever $\lambda>0$ is small enough.  The
		tubular coordinates $(t,h)$ and the rescaling
		$h=\lambda^{Mq_\nu}x$ are nonsingular for $\lambda>0$.  The two
		columns of the derivative of the original difference map are the tangent
		vectors of the two circle maps, with the second taken with the opposite sign.
		Hence the intersection determined by this zero is transverse.

		Summing over the $d$ components and using
		\Cref{eq:common-tangent-contact-sum} gives
		\[
			\sum_{\nu=1}^{d}2(\ell+k_\nu)
			=
			2m_im_j+2(Q_i\cdot Q_j)_c
		\]
		pairs of parameters contributing to
		$\dbl(A_{\mathfrak q_i,\lambda},A_{\mathfrak q_j,\lambda})$, and all
		these intersections are transverse.  Since
		$L\ne\overline L$, $Q_i$ and $\overline Q_j$ have distinct tangent
		lines, as do $\overline Q_i$ and $Q_j$.  Both intersection
		multiplicities are therefore $m_im_j$
		\cite[Lemma~4.4.1]{Wall2004SingularPoints}.  Conjugation gives
		$(Q_i\cdot Q_j)_c
		=(\overline Q_i\cdot\overline Q_j)_c$.  Summing the intersection
		multiplicities of the four pairs of branches gives
		\[
			\sum_{A\in\mathfrak q_i,\,B\in\mathfrak q_j}
			(A\cdot B)_c
			=
			2m_im_j+2(Q_i\cdot Q_j)_c.
		\]
		Summing $\dbl(A_{\mathfrak q_i,\lambda})$ over $i$ and
		$\dbl(A_{\mathfrak q_i,\lambda},A_{\mathfrak q_j,\lambda})$ over
		$i < j$ and applying \Cref{eq:delta-branch-formula} gives
		\[
			\dbl(A_{\mathcal P,\lambda})
			=
			\delta(C_{\mathcal P},c)-N.
		\]
		Since there are finitely many choices of $i<j$, decrease the bound on
		$\lambda$ so that
		\[
			\dbl(A_{\mathfrak q_i,\lambda},A_{\mathfrak q_j,\lambda})
			=
			2(Q_i\cdot Q_j)_c+2m_im_j
		\]
		for every $i<j$, and every intersection between two distinct circle maps
		is transverse.  Each
		circle map is an immersion. By 
		\Cref{thm:distinct-tangent-conjugate-pair-predivide}, if two
		different parameters with the same image, lie on one circle, their 
		tangent lines are distinct.  If they lie on
		different circles, their tangent lines are distinct because the zero of the
		difference map at their parameters is nondegenerate.  So we can 
		conclude that every multiple point of the
		disjoint union is ordinary.

		Choose a closed real disk $D\subset W\cap U_\R$ with
		$c\in\operatorname{int}D$.  The images of the circle maps converge uniformly
		to $c$.  Hence there is $\lambda_0>0$ such that, whenever
		$0<\lambda<\lambda_0$, each circle map is an immersion with ordinary
		multiple points, the formula for
		$\dbl(A_{\mathfrak q_i,\lambda},A_{\mathfrak q_j,\lambda})$ holds for
		every $i<j$,
		every intersection between the images of two distinct circles is transverse, and all
		the images lie in
		$\operatorname{int}D$.  The parameter
		circles have no boundary.  Therefore the
		boundary conditions in \Cref{def:predivide} hold, and the disjoint union of the
		circle maps parametrizes a predivide in $D$.
	\end{proof}

	\begin{lemma}
		\label{thm:distinct-tangent-pair-families}
		Let $X$ be a smooth complex surface with real structure $\sigma_X$. Let
		$c\in X_\R$, and let $\mathcal P$ be a nonempty finite set of conjugate
		pairs of branch germs at $c$.  Partition
		$\mathcal P=\mathcal P_1\sqcup\cdots\sqcup\mathcal P_s$, $s \ge 1$,
		into nonempty subsets so that the pairs in $\mathcal P_i$ have a common
		unordered pair of distinct tangent lines
		$\{T_i,\overline T_i\}$, where $T_i\ne\overline T_i$ and such that these
		unordered pairs are distinct for different $i$.  Label every
		$\mathfrak q=\{Q,\overline Q\}\in\mathcal P_i$ so that $T_cQ = T_i$.
		Choose adapted local data for
		$(X, c, \{T_1, \ldots, T_s\}, \mathcal P)$ and a positive integer $M$
		divisible by all the multiplicities $m_{\mathfrak q}$.
		Let $\mathcal I$ be a finite, possibly empty, index set.  For
		$\nu\in\mathcal I$, let
		\[
			\gamma_\nu:(\R,0)\longrightarrow(X_\R,c)
		\]
		be a real analytic map germ with
		$\gamma_\nu(0)=c$ and $\gamma_\nu'(0)\ne0$, and denote the germ of its
		image by $\Gamma_\nu$.  Assume that the $\Gamma_\nu$ are 
		transverse, i.e. they have pairwise
		distinct tangent lines, at $c$.

		Then there are positive numbers $\kappa_1,\ldots,\kappa_s$ such that,
		for every neighborhood $W$ of $c$ in $X_\R$, there are
		$\tau_0 > 0$ and a closed real disk $D\subset W\cap U_\R$, with
		$c\in\operatorname{int}D$, such that each $\Gamma_\nu$ has an
		embedded compact interval representating it in $D$ (that we  denote by 
		the same symbol), and
		these representatives meet pairwise only at $c$.  Whenever
		$0<\tau<\tau_0$, put
		\[
			A_{\mathfrak q,\tau}
			=
			\alpha^{\mathrm{cp}}_{\boldsymbol\gamma_{\mathfrak q},
			\kappa_i^{1/m_{\mathfrak q}}\tau^{M/m_{\mathfrak q}}}
			(\R/2\pi\Z)
			\subset\operatorname{int}D
			\qquad
			(\mathfrak q\in\mathcal P_i).
		\]
		The corresponding circle maps and the inclusions
		$\Gamma_\nu\hookrightarrow D$, $\nu\in\mathcal I$, jointly
		parametrize a predivide in $D$. If we denote $C_{\mathfrak 
		q}=Q\cup\overline Q$ for
		$\mathfrak q = \{Q, \overline Q\} \in \mathcal P$.  Then
		$c\notin A_{\mathfrak q,\tau}$ for every
		$\mathfrak q\in\mathcal P$, and
		\[
		\begin{aligned}
			\dbl(A_{\mathfrak q,\tau})
			&=\delta(C_{\mathfrak q},c)-1
			&&(\mathfrak q\in\mathcal P),\\
			\dbl(A_{\mathfrak q,\tau},A_{\mathfrak p,\tau})
			&=\sum_{A\in\mathfrak q,\,B\in\mathfrak p}(A\cdot B)_c
			&&(\mathfrak q,\mathfrak p\in\mathcal P,\
			\mathfrak q\ne\mathfrak p),\\
			\dbl(A_{\mathfrak q,\tau},A_{\mathfrak p,\tau})
			&=4m_{\mathfrak q}m_{\mathfrak p}
			&&(\mathfrak q\in\mathcal P_i,\
			\mathfrak p\in\mathcal P_j,\ i\ne j),\\
			\dbl(\Gamma_\nu,A_{\mathfrak q,\tau})
			&=2m_{\mathfrak q}
			=\mt(C_{\mathfrak q},c)
			&&(\nu\in\mathcal I,\ \mathfrak q\in\mathcal P).
		\end{aligned}
		\]
		Put $A_{\mathcal P,\tau}
		=\bigcup_{\mathfrak q\in\mathcal P}A_{\mathfrak q,\tau}$ and
		$C_{\mathcal P}
		=\bigcup_{\mathfrak q\in\mathcal P}C_{\mathfrak q}$.
		With respect to the disjoint union of the circle maps,
		\[
			\dbl(A_{\mathcal P,\tau})
			=
			\delta(C_{\mathcal P},c)-|\mathcal P|.
		\]
	\end{lemma}

	\begin{proof}
		We use the adapted local data fixed in the statement (recall definition 
		and notation at the beginning of the section). In particular, we work 
		in the 
		coordinates
		$\vartheta$, we identify $X_\R$ with $\R^2$, and we write
		$H_i=H_{T_i}$.  Put
		\[
			S_\tau=\tau^M.
		\]
		For $i=1,\ldots,s$, define the real linear isomorphism
		\[
			L_i^0:\C\longrightarrow\R^2,
			\qquad
			L_i^0(z)=H_i^{-1}(z,\overline z).
		\]
		For $\mathfrak q\in\mathcal P_i$, use the parameter
		$\tau^{M/m_{\mathfrak q}}$.  Take formula in
		\Cref{eq:common-tangent-normalized-circle} expressed  in the adapted
		coordinates $\chi_{T_i}=H_i\circ\vartheta$. Then apply
		$H_i^{-1}$ to it. This gives the convergence $
			S_\tau^{-1}
			\vartheta\circ
			\alpha^{\mathrm{cp}}_{\boldsymbol\gamma_{\mathfrak q},
			\tau^{M/m_{\mathfrak q}}}(\theta)
			\longrightarrow
			L_i^0(e^{im_{\mathfrak q}\theta}).
		$
		This is a convergence in $C^1$, as $\tau\to0^+$.  Put
		$\kappa_i=|\det_{\R}L_i^0|^{-1/2}$ and
		$L_i=\kappa_iL_i^0$.  Thus
		$|\det_{\R}L_i|=1$.  Apply
		\Cref{thm:common-tangent-annuli} to
		$\mathcal P_i$, using the same integer $M$ and the parameter
		$\mu=\kappa_i^{1/M}\tau$.  Since
		$\mu^M=\kappa_iS_\tau$, the radius of the parameter circle for
		$\mathfrak q\in\mathcal P_i$ is then
		\[
			\kappa_i^{1/m_{\mathfrak q}}\tau^{M/m_{\mathfrak q}}
		\]
		and, after division by $S_\tau$, the circle map converges in $C^1$ to
		$\theta\mapsto L_i(e^{im_{\mathfrak q}\theta})$.
		\Cref{thm:common-tangent-annuli} also computes
		$\dbl(A_{\mathfrak q, \tau})$ and
		$\dbl(A_{\mathfrak q,\tau},A_{\mathfrak p,\tau})$ when
		$\mathfrak q,\mathfrak p\in\mathcal P_i$.

		Denote $a_i=H_i^{-1}(1,0)$.  Then, compatibility with the real 
		structures gives
		\[
			L_i^0(z)=za_i+\overline z\,\overline{a_i},
			\qquad
			\C a_i=d\vartheta_c(T_i).
		\]
		Fix $i \ne j$ and put $B=L_j^{-1}L_i$.  Then
		$|\det_{\R}B|=1$, and
		$B$ is not orthogonal.  Indeed, an orthogonal real map of $\C$ has the
		form $z\mapsto e^{i\phi}z$ or $z\mapsto e^{i\phi}\overline z$.  Either
		form would give
		$\{T_i, \overline T_i\} = \{T_j, \overline T_j\}$.  The singular values of
		$B$ are therefore $\sigma$ and $\sigma^{-1}$ for some $\sigma>1$.
		After an orthogonal change of coordinates,
		\[
			\|B(e^{it})\|^2
			=
			\sigma^2\cos^2t+\sigma^{-2}\sin^2t.
		\]
		So the equation $\|B(e^{it})\|=1$ has four solutions.  
		In other words, the
		ellipses $L_i(S^1)$ and $L_j(S^1)$ meet in four points transversely.  If
		$\mathfrak q \in \mathcal P_i$ and $\mathfrak p\in\mathcal P_j$, each of
		these points has $m_{\mathfrak q}m_{\mathfrak p}$ preimages in the product
		of the two parameter circles.  Hence the map
		\[
			(\theta,\eta)\longmapsto
			L_i(e^{im_{\mathfrak q}\theta})
			-
			L_j(e^{im_{\mathfrak p}\eta})
		\]
		has $4m_{\mathfrak q}m_{\mathfrak p}$ nondegenerate zeros.   The norm 
		of the map above is
		bounded away from zero outside a small neighborhood of these zeros.  
		So 
		the implicit function theorem yields
		$4m_{\mathfrak q}m_{\mathfrak p}$ nondegenerate zeros for every sufficiently
		small $\tau>0$, that is,  $\dbl(A_{\mathfrak q,\tau},A_{\mathfrak 
		p,\tau})
		=4m_{\mathfrak q}m_{\mathfrak p}$.
		Each of the four pairs of branches
		in $\mathfrak q\times\mathfrak p$ has distinct
		tangent lines, so its intersection multiplicity is
		$m_{\mathfrak q}m_{\mathfrak p}$.  Therefore
		\[
			\sum_{A\in\mathfrak q,\,B\in\mathfrak p}
			(A\cdot B)_c
			=
			4m_{\mathfrak q}m_{\mathfrak p}.
		\]
		Fix $\mathfrak q\in\mathcal P_i$.
		The ellipse $L_i(S^1)$ does not contain the origin.  Hence there are constants
		$a,b>0$ such that
		\[
			aS_\tau\le |y-c|\le bS_\tau
			\qquad
			(y\in A_{\mathfrak q,\tau}).
		\]
		Thus $c\notin A_{\mathfrak q,\tau}$. 
		For $\nu\in\mathcal I$, put $v_\nu=\gamma_\nu'(0)$.
		Since
		\[
			\gamma_\nu(u)
			=
			c+uv_\nu+O(u^2)
		\]
		and $v_\nu\ne0$, after decreasing the interval on which
		$\gamma_\nu$ is defined there is $a_\nu>0$ such that
		\[
			|\gamma_\nu(u)-c|
			\ge
			a_\nu|u|.
		\]
		Hence every intersection with $A_{\mathfrak q,\tau}$ has
		$|u|=O(S_\tau)$.  Set $u=S_\tau x$.  On a compact interval containing
		all the resulting values of $x$,
		\[
			S_\tau^{-1}
			\bigl(\gamma_\nu(S_\tau x)-c\bigr)
			=
			xv_\nu+O(S_\tau x^2)
			\longrightarrow
			xv_\nu
		\]
		in $C^1$.  The line $\R v_\nu$ meets $L_i(S^1)$ twice and
		transversely.  Since
		$\theta\mapsto L_i(e^{im_{\mathfrak q}\theta})$ covers the ellipse
		$m_{\mathfrak q}$ times, the map
		\[
			(x,\theta)\longmapsto
			xv_\nu-L_i(e^{im_{\mathfrak q}\theta})
		\]
		has $2m_{\mathfrak q}$ nondegenerate zeros.  Convergence in $C^1$
		and the implicit function theorem give exactly
		$2m_{\mathfrak q}$ nondegenerate zeros for small
		$\tau>0$.  Hence
		\[
			\dbl(\Gamma_\nu,A_{\mathfrak q,\tau})
			=
			2m_{\mathfrak q},
		\]
		and, moreover, all these intersections are transverse.
		Since $C_{\mathfrak q}$ is the union of two branches of
		multiplicity $m_{\mathfrak q}$, additivity of multiplicity gives
		$\mt(C_{\mathfrak q},c)=2m_{\mathfrak q}$.

		After (possibly) shrinking the intervals on which the maps $\gamma_\nu$ 
		are
		defined, each map is an embedding.  Since the tangent lines are pairwise
		distinct, the arcs meet transversely only
		at $c$ near $c$. Fix a neighborhood $W$ of $c$.  Choose a closed
		real disk $D\subset W\cap U_\R$ with $c\in\operatorname{int}D$. Assume 
		that $D$ is small enough so  that each image meets $D$ in one embedded 
		interval and meets
		$\partial D$ transversely at its two endpoints.  Denote these intervals
		by $\Gamma_\nu$.

		Since the images of the circle maps shrink to $c$, we can choose
		$\tau_0>0$ so that: the formulas for
		$\dbl(A_{\mathfrak q,\tau})$,
		$\dbl(A_{\mathfrak q,\tau},A_{\mathfrak p,\tau})$ and
		$\dbl(\Gamma_\nu,A_{\mathfrak q,\tau})$ hold, every intersection
		counted by these formulas is transverse, and all the images of the circles lie in
		$\operatorname{int}D$ when $0<\tau<\tau_0$.  On each
		$\mathcal P_i$, we have that
		\Cref{thm:common-tangent-annuli} implies that every multiple point  is 
		ordinary.
		At an intersection of images of circles from different sets $\mathcal P_i$, or
		of the image of a circle and an arc, the two tangent lines are distinct by the
		nondegeneracy of the corresponding zero of the difference map shown 
		before.  At $c$,
		all the arcs have distinct tangents and none of the images of the 
		circles
		contains $c$.
		So every multiple point is ordinary.  And so the circle maps and the 
		inclusions
		$\Gamma_\nu\hookrightarrow D$ jointly parametrize a predivide. Finally, 
		summing the values of
		$\dbl(A_{\mathfrak q, \tau})$ and
		$\dbl(A_{\mathfrak q,\tau},A_{\mathfrak p,\tau})$ for all pairs
		in $\mathcal P$, and applying \Cref{eq:delta-branch-formula}, gives
		\[
			\dbl(A_{\mathcal P,\tau})
			=
			\delta(C_{\mathcal P},c)-|\mathcal P|.
		\]
	\end{proof}

	\begin{theorem}
		\label{thm:global-predivide}
		Every reduced real plane curve germ  $(C,0)$ admits a real
		morsification. More concretely, let $f\in\R\{x,y\}$ be a reduced
		equation of the real plane curve germ
		$(C, 0)$, and let $\mathbb B$ be a sufficiently small Milnor ball for
		$f$, invariant by complex conjugation.  Then there are a disk
		$\Delta_\lambda$, invariant under complex conjugation, and a
		holomorphic function $F$, defined on a neighborhood of
		$\overline{\mathbb B}\times\Delta_\lambda$, such that
		$F(\overline x, \overline y, \overline\lambda)
		=\overline{F(x,y,\lambda)}$.  Moreover, $F(x,y,0)=f(x,y)$, and
		$\{F=0\}$ is flat over $\Delta_\lambda$.
		For every sufficiently small nonzero $\lambda$, the fiber
		$\{F(\cdot,\lambda)=0\}\cap\mathbb B$ has exactly
		$\delta(C, 0)-\imbr(C, 0)$ singular points, all real
		hyperbolic nodes. And so 
		$P_\lambda = \{F(\cdot, \lambda) = 0\}\cap\overline{\mathbb B}_\R$ is a
		divide with
		\[
			\dbl(P_\lambda)=\delta(C,0)-\imbr(C,0).
		\]
		Thus, the function $F$ defines a real morsification of $f$.
	\end{theorem}

	\begin{proof}
		\noindent\emph{Setup.}
		Choose an invariant Milnor ball $\mathbb B^+$ such that
		$\overline{\mathbb B}\subset\mathbb B^+$, and take the sequence of real
		blowups
		\[
			\rho:X_N\longrightarrow\cdots\longrightarrow X_0=\C^2
		\]
		from step~\ref{it:global-algorithm-resolution} of the
		\hyperref[roadmap:global-construction]{Roadmap}.  The real structure 
		lifts to a real structure in $X_j$ after the blow up $X_j \to X_{j-1}$
		(because we are blowing up a real point). Every exceptional
		component created over a real center is invariant by this real 
		structure.  For every real center
		$z$, denote by $E_z$ the exceptional component created by blowing up 
		$z$ (we denote by the same symbol its strict transforms after further 
		blowups). We write
		$E_{z,\R}$ for its real part.  For an orbit 
		$\mathfrak o\in\Omega$ and a real center $z$, we define the number
		\[
			M_{\mathfrak o}(z)
			=
			\sum_{B\in\mathfrak o}\mt(B^{(z)},z),
		\]
		where a branch that does not pass through $z$ contributes zero.  Choose
		t primitive parametrizations of the each of the (smooth) real strict 
		transforms of the real branches of $(C,0)$.

		Throughout the proof, we need to follow holomorphically the preimages of
		the intersections already constructed. In particular, at the end, we 
		need to track the two
		preimages of each node. These is needed to be able to conclude that the 
		images of the disks and cylinders by our parametrizations have an 
		analytic structure. These points may initially form multisections
		rather than separate holomorphic sections. In order to deal with this, 
		we make a finite number of
		changes of parameter as follows.  Start with $t_0=t$.  At stage $j$,
		work with the current parameter $t_{j-1}$.  After shrinking its disk, 
		the relevant multisections are unramified away from the origin.  Let 
		$d_j$ be
		a common multiple of their ramification indices and set
		$t_{j-1}=t_j^{d_j}$.  Consider the family obtained by this base change 
		and
		replace its total space by its normalization.  Rewrite all the
		parametrizations and perturbations from before in terms of $t_j$ by
		substituting $t_{j-1}=t_j^{d_j}$.   The preimages of the multiple 
		points that have been constructed so far, become
		holomorphic sections.  At the end, we do the same for the (two) 
		preimages
		of each new node.

		There are only finitely many stages and hence finitely many base changes,
		say $a$.  At the end, put
		$\lambda=t_a$ and $d=d_1\cdots d_a$, with $d=1$ if no base change is
		required.  Their composition is
		\begin{equation}\label{eq:global-base-change}
			b:\Delta_\lambda\longrightarrow\Delta_t,
			\qquad
			t=b(\lambda)=\lambda^d.
		\end{equation}
		The map \cref{eq:global-base-change} is equivariant with respect to the 
		real 
		structures. In particular, the resulting 
		deformation is real.

		\noindent\emph{Descent and the lower bound.}
		We contract the exceptional components in reverse order and make sure 
		that the
		following condition is preserved.  Just before $E_z$ is 
		contracted, the real
		locus constructed for $\mathfrak o$ meets $E_{z,\R}$ in
		$M_{\mathfrak o}(z)$ distinct transverse real points away from the
		intersection points of $E_{z}$ with the other exceptional divisor.  All 
		multiple points already
		constructed are ordinary and lie away from the exceptional divisor.  
		This condition holds trivially on $X_N$ by the choice of the
		resolution.

		Suppose that $\beta:X_k\to X_{k-1}$ contracts $E_z$ to $c$.  In local
		coordinates, $\beta(u,v)=(u,uv)$ and $E_z=\{u=0\}$.  So the germs that 
		are
		transverse to $E_z$ at distinct real points descend in the blowdown to 
		smooth real germs
		through $c$ with distinct tangent lines. Moreover,  they are also 
		transverse to the
		real exceptional divisors that pass through $c$.  Recall
		in step~\ref{it:global-algorithm-assembly} of the
		\hyperref[roadmap:global-construction]{Roadmap}.

		Consider the packet $\mathcal P_c$ from
		step~\ref{it:global-algorithm-packets}.  If it is empty, no circle is
		introduced at $c$.  Otherwise, the two tangent lines of every pair in
		$\mathcal P_c$ are distinct and nonreal.  Indeed, otherwise a common 
		tangent would
		be real, and the pair would have another common real infinitely near
		point.  Partition $\mathcal P_c$ into its unordered pairs of conjugate
		tangent lines.  For
		$\mathfrak q=\{Q,\overline Q\}\in\mathcal P_c$, set
		$m_{\mathfrak q}=\mt(Q^{(c)},c)$, and take
		\[
			M_c=\operatorname{lcm}
			\{m_{\mathfrak q}:\mathfrak q\in\mathcal P_c\}.
		\]
		Apply \Cref{thm:distinct-tangent-pair-families} with $M=M_c$ and
		$\tau=\lambda^{N_c}$, where $N_c$ is chosen after all earlier scales 
		have already been chosen.
		Among the real arcs in the statement of that theorem, we include the 
		arcs corresponding 
		to the real part all exceptional divisors
		passing through $c$.

		If $\mathcal P_c\ne\varnothing$, taking $N_c$ sufficiently large
		puts the images of the new circle maps in a
		neighborhood of $c$ disjoint from the multiple points that have been 
		already produced. The conclusions of
		\Cref{thm:distinct-tangent-pair-families} hold uniformly for small
		$\lambda$. That is, the new circles and real arcs form a local 
		predivide, all new
		intersections are transverse, and a pair of multiplicity
		$m_{\mathfrak q}$ meets each real arc through $c$ in
		$2m_{\mathfrak q}$ points.

		Before the next contraction, we translate as in
		step~\ref{it:global-algorithm-translations}, with coefficient
		$\lambda^{P_\ell}$ for a new sufficiently large integer $P_\ell$.
		Consider the marked sections obtained in the construction leading to
		\Cref{eq:global-base-change}. 
		What we sketch now is the relative form of A'Campo's
		procedure for these translations
		\cite[\S~2, pp.~6--11]{ACampo1975MonodromyDeploiement}. Here we 
		are actually using the
		parametrized version from in
		\cite[\S~2.1.2, especially Remark~2.1.19]
		{Castellini2015AcampoDeformations} and also
		\cite[\S~2.1.1(1) and \S~2.3]
		{LeviantShustin2018Morsifications}. See those references for more 
		details.
		On each disk or annulus, we perturb the parametrization by
		\[
		\gamma_\lambda(s)
		\longmapsto
		\gamma_\lambda(s)+\lambda^{P_\ell}g_\lambda(s)v_\ell,
		\]
		where $v_\ell$ is a real vector and $g_\lambda$ vanishes to
		order at least two at the preimages of the multiple points that have 
		been already
		constructed.  These points and their tangent directions are therefore
		unchanged.  We make sure to choose $g_\lambda$ nonzero near the 
		remaining exceptional
		divisors, and also choose $v_\ell$ generic enough. This assures that  
		the new intersections
		with that divisor are real, transverse and pairwise distinct.

To finish, we just need to check the number of intersection points required in 
the
		induction.  If an
		exceptional component $E_w$ of $X_{k-1}$ passes through $c$, then
		\begin{equation}\label{eq:exceptional-component-pullback}
			\beta^*E_w=\widetilde E_w+E_z.
		\end{equation}
		For a curve $D$ already constructed, the projection formula
		$(\beta_*D\cdot E_w)_{X_{k-1}}=(D\cdot\beta^*E_w)_{X_k}$, together with
		\cref{eq:exceptional-component-pullback}, gives that  after contracting 
		$E_z$
		we get  $(D\cdot E_z)_{X_k}$ more intersection points with $E_w$.  For 
		the curve
		that corresponds to $\mathfrak o$, this number is $M_{\mathfrak o}(z)$.
		Then, \Cref{thm:distinct-tangent-pair-families} applied with $E_{w,\R}$ 
		included among
		the real arcs, gives $2m_{\mathfrak q}$ intersection points for a
		conjugate pair introduced at $c$.  This  gives
		exactly $M_{\mathfrak o}(z')$ intersection points immediately before
		$E_{z'}$ is contracted, as required for the induction.

		After the last contraction, perform again the translations and 
		perturbations of
		step~\ref{it:global-algorithm-count} described above by the parametrized
		version of A'Campo's perturbation given by Castellini.  A
		small real
		perturbation, replaces an ordinary point of multiplicity $k$ by
		$\binom{k}{2}$ distinct real nodes and preserves the number of all 
		earlier transverse
		intersections.  Let $n$ be the total number of nodes obtained at the 
		end. 
		Let $\beta:X'\to X$ be the blowup of at $z$, and let
		$E=\beta^{-1}(z)$ be the corresponding exceptional divisor.  If
		$H\subset(X,z)$ is a curve germ, denote by $\widetilde H$  its strict 
		transform.  If $G$
		is reduced and $A,B\subset(X,z)$ have no common component, then we have 
		the following standard formulas
		\begin{equation}\label{eq:global-blowup-count}
		\begin{aligned}
			\delta(G,z)
			&=\binom{\mt(G,z)}2+\sum_{p\in E}\delta(\widetilde G,p),\\
			(A\cdot B)_z
			&=\mt(A,z)\mt(B,z)
			+\sum_{p\in E}(\widetilde A\cdot\widetilde B)_p.
		\end{aligned}
		\end{equation}
		These formulas can be found, for example, in
		\cite[Lemma~4.4.2 and Theorem~6.5.9]{Wall2004SingularPoints}.
		For $\mathfrak o\in\Omega$, put
		$C_{\mathfrak o}=\bigcup_{B\in\mathfrak o}B$.  Iterating
		\Cref{eq:global-blowup-count} and using
		\Cref{thm:distinct-tangent-pair-families} shows that an orbit consisting
		of a real branch contributes $\delta(C_{\mathfrak o},0)$ nodes, a
		conjugate pair
		contributes $\delta(C_{\mathfrak o},0)-1$, and two distinct orbits
		$\mathfrak o,\mathfrak p$ contribute
		$(C_{\mathfrak o}\cdot C_{\mathfrak p})_0$.  Therefore
		\Cref{eq:delta-branch-formula} gives
		\[
			n\geq\delta(C,0)-\imbr(C,0).
		\]

		\noindent\emph{Analytic realization.}
		For each orbit of branches $\mathfrak o\in\Omega$, consider the source 
		of its
		quotient trace map: a smooth germ $\mathcal A/G$ for a real branch and 
		the family $\mathcal A$ that smooths the node for a pair of complex conjugate 
		branches (recall in
		\Cref{eq:canonical-source-representative,eq:orbit-quotient-source}).  
		We 
		replace
		its parameter by $s=c_{\mathfrak o}\lambda^{e_{\mathfrak o}}$,
		where $c_{\mathfrak o}>0$ and
		$e_{\mathfrak o}\in\Z_{>0}$ are determined by the choices made in the 
		construction above.
		Denote this family by $\mathcal S_{\mathfrak o}$ and set
		\[
			\mathcal S=\coprod_{\mathfrak o\in\Omega}
			\mathcal S_{\mathfrak o}.
		\]
		 The real structures in
		\Cref{eq:canonical-source-real-structure} give a real structure
		$\sigma_{\mathcal S}$ on $\mathcal S$.  The quotient maps in
		\Cref{eq:quotient-trace-map} and
		the parametrizations constructed during the descent process of this 
		proof, define an
		equivariant holomorphic map
		\begin{equation}\label{eq:global-normalization-map}
			(\widetilde\Psi,\pi):
			\mathcal S\longrightarrow
			\mathbb B^+\times\Delta_\lambda,
		\end{equation}
		where equivariance follows from \Cref{eq:trace-equivariance-down}. The 
		central fiber $\mathcal S_0$ is a disjoint union of disks and nodal 
		singularities. So the
		normalization of $\mathcal S_0$ consists of one disk for each branch of
		$C$, and $\widetilde\Psi_0$ is the chosen primitive parametrization on
		each disk.

		Choose an invariant ball $U$ with
		$\overline{\mathbb B}\subset U$ and
		$\overline U\subset\mathbb B^+$.  Shrink the domains of the
		parametrizations so that
		\[
			(\widetilde\Psi,\pi):
			\widetilde\Psi^{-1}(U)\longrightarrow U\times\Delta_\lambda
		\]
		is proper, and replace $\mathcal S$ by $\widetilde\Psi^{-1}(U)$.  The map
		is finite because it is nonconstant on every component of every fiber.
		Remmert's proper mapping theorem
		\cite{Remmert1956Projektionen} gives an invariant analytic hypersurface
		\[
			Y=(\widetilde\Psi,\pi)(\mathcal S)
			\subset U\times\Delta_\lambda.
		\]
		Since every branch of $C$ occurs once and is parametrized primitively,
		\[
			(\widetilde\Psi_0)_*[\mathcal S_0]=[C].
		\]
		Hence $Y_0=C$ is reduced, and the finite map
		$\mathcal S\to Y$ has degree one over every irreducible component of $Y$.
		Since $\mathcal S$ is normal, this map is the normalization of $Y$.

		Since $U\times\Delta_\lambda$ is Stein and contractible, the 
		hypersurface $Y$ is defined by an $F$ and we can write
		$Y=\{F=0\}$.  Invariance under conjugation allows us to choose a real 
		$F$, that is, an $F$ such that
		\[
			F(\overline x,\overline y,\overline\lambda)
			=
			\overline{F(x,y,\lambda)}.
		\]
		Moreover, up to a unit, we may assume that  $F(\cdot,0)=f$.  In 
		particular,  $F$ 
		is not divisible by $\lambda$.  So $Y$ is flat over
		$\Delta_\lambda$.  After shrinking the parameter disk, every fiber is
		smooth near $\partial\mathbb B$ and transverse to it.

		\noindent\emph{The number of nodes and the real fibers.}
		 For $\lambda\ne0$, the
		fiber $\mathcal S_\lambda$ is smooth (a disjoint union of cylinders and 
		disks), and
		$\widetilde\Psi_\lambda:\mathcal S_\lambda\to Y_\lambda$ is the
		normalization of $Y_\lambda$.  After the base change
		$t=\lambda^d$ in \Cref{eq:global-base-change}, each of the $n$ nodes
		produced in the construction above, determines two holomorphic sections 
		of the source whose images agree for $\lambda>0$

		By construction, the fiber
		$\widetilde\Psi_\lambda^{-1}(\mathbb B) = S_\lambda $ consists of one 
		disk for each
		real branch and one annulus for each conjugate pair.  Since an annulus 
		has Euler characteristic equal to $0$, the Euler
		characteristic of $S_\lambda$ is $\rebr(C,0)$.  A small smoothing of
		$Y_\lambda\cap\mathbb B$ is isotopic to a Milnor fiber of $C$.  The
		Euler characteristic of this Milnor fiber is
		$\rebr(C,0)+2\imbr(C,0)-2\delta(C,0)$.  A simple Euler characteristic 
		computation, yields
		\[
			\sum_{q\in\operatorname{Sing}(Y_\lambda)\cap\mathbb B}
			\delta(Y_\lambda,q)
			=
			\delta(C,0)-\imbr(C,0).
		\]
		See
		\cite[Proposition~6.3.1, Theorems~6.4.1 and~6.5.9, and the subsequent
		remark]{Wall2004SingularPoints}.  Together with the lower bound proved
		above, this gives
		\[
			\delta(C,0)-\imbr(C,0)
			\leq n
			\leq
			\sum_{q\in\operatorname{Sing}(Y_\lambda)\cap\mathbb B}
			\delta(Y_\lambda,q)
			=
			\delta(C,0)-\imbr(C,0).
		\]
		Thus $n=\delta(C,0)-\imbr(C,0)$, and there are no other singular points
		in $\mathbb B$.

		For $\lambda>0$ small enough, by construction, the two preimages of each
		node lie in the locus fixed by the real structure (recall
		\Cref{eq:positive-real-source-circle,eq:quotient-distinguished-real-set}).
		Moreover, also by construction the nodes are actually real and
		hyperbolic.  Also by construction, the unique preimage of every smooth 
		real point of
		$Y_\lambda$ is fixed by the real structure.  Hence
		\[
			\{F(\cdot,\lambda)=0\}\cap\overline{\mathbb B}_\R
		\]
		is exactly the divide constructed during the descent and, as we have 
		already shown, has
		$\delta(C,0)-\imbr(C,0)$ double points.  The criterion of
		\Cref{rem:zero-level-convention} gives that $F$ defines a real
		morsification.
	\end{proof}

	\section{An example with seven branches}
	\label{sec:worked-seven-branch-example}

	This section applies the construction of \Cref{sec:global-predivide} to the
	following germ.

	\begin{example}
		\label{ex:seven-branch-global-construction}
		In real coordinates $(x, y)$, let $
			C=R\cup\bigcup_{i=1}^3(Q_i\cup\overline Q_i),
		$
		where
		$\mathfrak q_i = \{Q_i, \overline Q_i\}$ and $R$ has parametrization
		\[
			R: (x,y)=(t^2,t^3).
		\]
		The other branches are defined by the equations
		\[
			\begin{array}{c@{\qquad}c}
				Q_1:x=iy^2+y^3
				&\overline Q_1:x=-iy^2+y^3,\\
				Q_2:x=iy^2+2y^3
				&\overline Q_2:x=-iy^2+2y^3,\\
				Q_3:x=(2+3i)y^2
				&\overline Q_3:x=(2-3i)y^2.
			\end{array}
		\]
		A real equation for $C$ is
		\[
			f(x,y)
			=
			(y^2-x^3)
			\bigl((x-y^3)^2+y^4\bigr)
			\bigl((x-2y^3)^2+y^4\bigr)
			\bigl((x-2y^2)^2+9y^4\bigr).
		\]
	\end{example}

	The branch $R$ has delta invariant equal to one, and its intersection 
	multiplicity
	with each of the other branches is two.  Among the fifteen pairs of nonreal
	branches, thirteen have intersection multiplicity two.  The pairs
	$(Q_1,Q_2)$ and
	$(\overline Q_1, \overline Q_2)$ have intersection multiplicity three.
	Hence the formula \Cref{eq:delta-branch-formula} gives
	\[
		\delta(C,0)=1+6\cdot2+13\cdot2+2\cdot3=45,
		\qquad
		\delta(C,0)-\imbr(C,0)=45-3=42.
	\]

	For step~\ref{it:global-algorithm-resolution} of the
	\hyperref[roadmap:global-construction]{Roadmap}, blow up the origin
	$p_0$.  Write $X_1$ for the blowup and $E_1$ for its
	exceptional component.  The strict transform of $R$ meets $E_1$ at a
	point $p_R$ corresponding to the tangent line $\{y=0\}$.  The other six
	strict transforms meet $E_1$ at the point $c$ corresponding to the
	tangent line $\{x=0\}$.

	On $X_1$, choose local coordinates $(y,z)$ near $c$ in which the
	blowdown map is given by $(y,z)\longmapsto(yz,y)$.
	In these coordinates, we have that the exceptional component and the strict 
	transforms are defined by
	\[
		\begin{array}{c|c|c|c}
			E_1&Q_1,Q_2&\overline Q_1,\overline Q_2&Q_3,\overline Q_3\\ \hline
			y=0
			&z=iy+y^2,\ iy+2y^2
			&z=-iy+y^2,\ -iy+2y^2
			&z=(2+3i)y,\ z=(2-3i)y.
		\end{array}
	\]
	The two pairs of tangent lines at $c$ are
	\[
		\bigl\{\{z=iy\},\{z=-iy\}\bigr\}
		\quad\text{and}\quad
		\bigl\{\{z=(2+3i)y\},\{z=(2-3i)y\}\bigr\}.
	\]
 This gives the partition (see step~\ref{it:global-algorithm-packets}),
	\[
		\mathcal P_1=\{\mathfrak q_1,\mathfrak q_2\},
		\qquad
		\mathcal P_2=\{\mathfrak q_3\}.
	\]
	Let $E_P$ be the exceptional component of the blowup at $c$.  The
	strict transforms meet $E_P$ at four nonreal points. These are all the and  
	real blow ups needed.
	In the chart of $X_1$ containing $p_R$, choose coordinates
	$(x, w)$ in which the blowdown map to $X_0$ is $(x,w)\longmapsto(x,xw)$.
	The strict transform of $R$ has parametrization
	$(x, w) = (t^2, t)$.  Blow up $p_R$, and let $E_R$ be the exceptional
	component.  Blow up $p_S=E_1\cap E_R$, and denote the new exceptional
	component by $E_S$.  The centers of these sequence of blow ups are
	$p_0, p_R,c$ and $p_S$.
	The dual graph can be seen in
	\Cref{fig:seven-branch-resolution-tree}.

	\begin{figure}[!ht]
		\centering
		\tikzsetnextfilename{seven-branch-resolution-tower}
		\resizebox{0.68\linewidth}{!}{%
			\begin{tikzpicture}[
	x=1cm,
	y=1cm,
	line cap=round,
	line join=round,
	>=stealth
]
\definecolor{orbitone}{HTML}{D55E00}
\definecolor{orbittwo}{HTML}{0072B2}
\definecolor{orbitthree}{HTML}{009E73}
\definecolor{realorbit}{HTML}{7A3E9D}
\path[use as bounding box] (-.25,-3.03) rectangle (13.48,3.03);
\tikzset{
	exceptionalvertex/.style={circle,draw=black,fill=white,line width=1pt,
		minimum size=9pt,inner sep=0pt},
	nonreal/.style={circle,draw=black,fill=white,line width=.85pt,
		minimum size=7.2pt,inner sep=0pt},
	graph edge/.style={black,line width=1pt},
	realbranch/.style={realorbit,line width=1.55pt},
	qone/.style={orbitone,line width=1.35pt},
	qtwo/.style={orbittwo,line width=1.35pt},
	qthree/.style={orbitthree,line width=1.35pt}
}

\coordinate (eR) at (.75,0);
\coordinate (eS) at (2.85,0);
\coordinate (e1) at (4.95,0);
\coordinate (eP) at (7.15,0);
\draw[graph edge] (eR) -- (eS) -- (e1) -- (eP);

\foreach \pos/\weight in {eR/-2,eS/-1,e1/-4,eP/-1}
{
	\node[exceptionalvertex] at (\pos) {};
	\node[font=\small,below=7pt] at (\pos) {\((\weight)\)};
}
\node[font=\small,above=7pt] at (eR) {\(E_R\)};
\node[font=\small,anchor=east] at (2.70,.32) {\(E_S\)};
\node[font=\small,above=7pt] at (e1) {\(E_1\)};
\node[font=\small,above=7pt] at (eP) {\(E_P\)};

\draw[realbranch,->] (eS) -- (2.85,2.05);
\node[font=\small,text=realorbit,anchor=south] at (2.85,2.09)
	{\(R:(t^2,t^3)\)};

\coordinate (uplus)  at (9.25,1.18);
\coordinate (uminus) at (9.25,-1.18);
\coordinate (vplus)  at (11.35,1.18);
\coordinate (vminus) at (11.35,-1.18);
\foreach \p in {uplus,uminus,vplus,vminus}
	\draw[graph edge] (eP) -- (\p);

\node[nonreal,draw=orbittwo,minimum size=10pt] at (uplus) {};
\node[nonreal,draw=orbitone,minimum size=5pt] at (uplus) {};
\node[nonreal,draw=orbittwo,minimum size=10pt] at (uminus) {};
\node[nonreal,draw=orbitone,minimum size=5pt] at (uminus) {};
\node[nonreal,draw=orbitthree] at (vplus) {};
\node[nonreal,draw=orbitthree] at (vminus) {};

\draw[qone,->] (uplus) -- (8.68,2.35);
\draw[qtwo,->] (uplus) -- (9.82,2.35);
\draw[qone,->] (uminus) -- (8.68,-2.35);
\draw[qtwo,->] (uminus) -- (9.82,-2.35);
\draw[qthree,->] (vplus) -- (12.32,2.35);
\draw[qthree,->] (vminus) -- (12.32,-2.35);

\node[font=\small,text=orbitone,anchor=south east] at (8.68,2.39)
	{\(Q_1\)};
\node[font=\small,text=orbittwo,anchor=south west] at (9.82,2.39)
	{\(Q_2\)};
\node[font=\small,text=orbitone,anchor=north east] at (8.68,-2.39)
	{\(\overline Q_1\)};
\node[font=\small,text=orbittwo,anchor=north west] at (9.82,-2.39)
	{\(\overline Q_2\)};
\node[font=\small,text=orbitthree,anchor=south west] at (12.32,2.39)
	{\(Q_3\)};
\node[font=\small,text=orbitthree,anchor=north west] at (12.32,-2.39)
	{\(\overline Q_3\)};

\end{tikzpicture}%
		}
		\caption{The dual graph of the resolution.  The purple arrow represents
		the strict transform of $R$.  The other six arrows at $E_P$ represent
		the strict transforms of the nonreal branches, which meet $E_P$ at
		four points.}
		\label{fig:seven-branch-resolution-tree}
	\end{figure}
	\FloatBarrier

	 As in
step~\ref{it:global-algorithm-initialization}, choose a primitive
parametrization of the strict transform of $R$ from a smooth disk.
Let $J_R$ denote the image of a compact interval conatained in the locus 
fixed by the 
real structure of
that disk.  Figure (a) shows $J_R$ near the exceptional components over
$p_R$, immediately before $E_S$ is contracted.  Its intersection with
$E_S$ is far from $E_S\cap E_R$ and $E_S\cap E_1$.
Now we contract $E_S$. Then, the
components $E_R$ and $E_P$ in either order, and finally
$E_1$.
A small translation, as in
step~\ref{it:global-algorithm-translations}, makes the intersections of
$J_R$ with $E_1$ and $E_R$ distinct and transverse (see figure~(b)).
	\definecolor{orbitone}{HTML}{D55E00}
\definecolor{orbittwo}{HTML}{0072B2}
\definecolor{orbitthree}{HTML}{009E73}
\definecolor{realorbit}{HTML}{7A3E9D}
\definecolor{nodecharcoal}{HTML}{252525}

\tikzset{
	descent exceptional/.style={black!68,line width=1.05pt},
	descent realbranch/.style={realorbit,solid,line width=.95pt},
	descent rawbranch/.style={realorbit!42,solid,line width=.65pt},
	descent qone/.style={orbitone,solid,line width=.82pt},
	descent qtwo/.style={orbittwo,solid,line width=.82pt},
	descent qthree/.style={orbitthree,solid,line width=.82pt},
	descent translation/.style={->,realorbit!82,line width=.72pt},
	descent corner/.style={circle,draw=black!72,fill=white,line width=.6pt,
		minimum size=3.6pt,inner sep=0pt},
	descent limiting/.style={circle,draw=black!72,fill=white,line width=.6pt,
		minimum size=3.2pt,inner sep=0pt},
	descent purplehit/.style={circle,draw=white,fill=realorbit,line width=.5pt,
		minimum size=3.5pt,inner sep=0pt},
	descent orangehit/.style={circle,draw=white,fill=orbitone,line width=.5pt,
		minimum size=3.5pt,inner sep=0pt},
	descent bluehit/.style={circle,draw=white,fill=orbittwo,line width=.5pt,
		minimum size=3.5pt,inner sep=0pt},
	descent greenhit/.style={circle,draw=white,fill=orbitthree,line width=.5pt,
		minimum size=3.5pt,inner sep=0pt},
	descent node/.style={circle,draw=white,fill=nodecharcoal,line width=.35pt,
		minimum size=2.6pt,inner sep=0pt}
}


\newcommand{\SevenBranchDescentA}{%
	\begin{tikzpicture}[x=1cm,y=1cm,line cap=round,line join=round,>=stealth]
		\path[use as bounding box] (-.10,-.10) rectangle (6.30,3.05);
		\def\yy{1.48}

		\draw[descent exceptional] (.45,\yy) -- (5.95,\yy);
		\draw[descent exceptional] (1.02,.32) -- (1.02,2.64);
		\draw[descent exceptional] (5.38,.32) -- (5.38,2.64);
		\draw[descent realbranch]
			(3.02,2.63) .. controls (3.10,2.20) and (3.28,1.76) ..
			(3.34,\yy) .. controls (3.40,1.12) and (3.60,.65) ..
			(3.76,.33);

		\node[descent corner] at (1.02,\yy) {};
		\node[descent purplehit] at (3.34,\yy) {};
		\node[descent corner] at (5.38,\yy) {};

		\node[font=\scriptsize,anchor=south east] at (.95,2.58) {\(E_1\)};
		\node[font=\scriptsize,anchor=south west] at (5.46,2.58) {\(E_R\)};
		\node[font=\scriptsize,anchor=south] at (4.40,1.58) {\(E_S\)};
		\node[font=\scriptsize,text=realorbit,anchor=west] at (3.82,.55)
			{\(J_R\)};
	\end{tikzpicture}%
}

\newcommand{\SevenBranchDescentB}{%
	\begin{tikzpicture}[x=1cm,y=1cm,line cap=round,line join=round,>=stealth]
		\path[use as bounding box] (-.10,-.10) rectangle (6.30,3.05);

		\draw[descent exceptional] (1.77,.27) -- (1.77,2.66);
		\draw[descent exceptional] (.42,1.20) -- (6.00,1.20);
		\node[descent corner] at (1.77,1.20) {};

		\draw[descent rawbranch] (.82,.30) -- (3.18,2.66);
		\draw[descent realbranch] (.55,.58) -- (2.82,2.85);
		\node[descent purplehit] at (1.17,1.20) {};
		\node[descent purplehit] at (1.77,1.80) {};

		\draw[descent translation] (2.55,1.90) -- (2.55,2.42);
		\node[font=\scriptsize,text=realorbit,anchor=west] at (2.65,2.16)
			{\(\lambda^{P_S}\)};

		\node[font=\scriptsize,anchor=south east] at (1.70,2.60) {\(E_1\)};
		\node[font=\scriptsize,anchor=south] at (5.05,1.30) {\(E_R\)};
	\end{tikzpicture}%
}

\newcommand{\SevenBranchDescentC}{%
	\begin{tikzpicture}[x=1cm,y=1cm,line cap=round,line join=round,>=stealth]
		\path[use as bounding box] (1.78,-.10) rectangle (4.48,3.05);
		\def\yy{1.48}

		\draw[descent exceptional] (2.55,.22) -- (2.55,2.73);
		\draw[descent realbranch,samples=121,smooth,variable=\vv,
			domain=.35:2.61]
			plot ({2.28+1.04*(\vv-\yy)*(\vv-\yy)},\vv);
		\pgfmathsetmacro{\cuspdeltaC}{sqrt(.27/1.04)}
		\node[descent purplehit] at (2.55,{\yy-\cuspdeltaC}) {};
		\node[descent purplehit] at (2.55,{\yy+\cuspdeltaC}) {};
		\node[descent limiting] at (2.55,\yy) {};

		\node[font=\scriptsize,anchor=south east] at (2.48,2.67) {\(E_1\)};
		\node[font=\scriptsize,anchor=east] at (2.32,\yy) {\(p_R\)};
		\node[font=\scriptsize,text=realorbit,anchor=west] at (3.82,\yy)
			{\(J_R\)};
	\end{tikzpicture}%
}

\newcommand{\SevenBranchDescentD}{%
	\begin{tikzpicture}[x=1cm,y=1cm,line cap=round,line join=round,>=stealth]
		\path[use as bounding box] (2.34,-.10) rectangle (4.30,3.05);
		\def\pR{.83}
		\def\cc{2.18}

		\draw[descent exceptional] (3.00,.13) -- (3.00,2.85);
		\node[font=\scriptsize,anchor=north east] at (2.94,2.98) {\(E_1\)};

		\draw[descent realbranch,samples=101,smooth,variable=\vv,
			domain=.28:1.38]
			plot ({2.73+1.58*(\vv-\pR)*(\vv-\pR)},\vv);
		\pgfmathsetmacro{\cuspdeltaD}{sqrt(.27/1.58)}
		\node[descent purplehit] at (3.00,{\pR-\cuspdeltaD}) {};
		\node[descent purplehit] at (3.00,{\pR+\cuspdeltaD}) {};
		\node[descent limiting] at (3.00,\pR) {};

		\begin{scope}[shift={(3.00,\cc)}]
			\fill[white] (0,0) circle[radius=.22];
			\fill[orbitone] (0,0) -- (90:.18) arc (90:210:.18) -- cycle;
			\fill[orbittwo] (0,0) -- (210:.18) arc (210:330:.18) -- cycle;
			\fill[orbitthree] (0,0) -- (330:.18) arc (330:450:.18) -- cycle;
			\draw[black!55,line width=.25pt] (0,0) circle[radius=.18];
		\end{scope}

		\node[font=\scriptsize,anchor=west] at (3.82,\pR) {\(p_R\)};
		\node[font=\scriptsize,anchor=west] at (3.37,\cc) {\(c\)};
	\end{tikzpicture}%
}

\newcommand{\SevenBranchDescentE}{%
	\begin{tikzpicture}[x=1cm,y=1cm,line cap=round,line join=round,>=stealth]
		\path[use as bounding box] (2.15,-.10) rectangle (4.30,3.05);
		\def\packetoffset{.205}
		\def\pR{.66}
		\def\cc{2.12}
		\def\circlescale{.55}
		\def\mixedpacketnodes{%
			.729007/.622006,.553523/.741050,-.918680/.103388,
			-.889163/-.405888,.889163/.405888,-.729007/-.622006,
			-.553523/-.741050,.918680/-.103388}

		\draw[descent exceptional] (3.00,.05) -- (3.00,2.90);
		\node[font=\scriptsize,anchor=north west] at (3.08,2.98) {\(E_1\)};

		\draw[descent realbranch,samples=101,smooth,variable=\vv,
			domain=.17:1.15]
			plot ({2.73+1.58*(\vv-\pR)*(\vv-\pR)},\vv);
		\pgfmathsetmacro{\cuspdeltaE}{sqrt(.27/1.58)}
		\node[descent purplehit] at (3.00,{\pR-\cuspdeltaE}) {};
		\node[descent purplehit] at (3.00,{\pR+\cuspdeltaE}) {};
		\node[descent limiting] at (3.00,\pR) {};

		\draw[descent qone]
			plot[domain=0:360,samples=281,smooth cycle,variable=\a]
			({3.00+\circlescale*((1+.08*cos(3*\a))*cos(\a)-\packetoffset)},
			 {\cc+\circlescale*(1+.08*cos(3*\a))*sin(\a)});
		\draw[descent qtwo]
			plot[domain=0:360,samples=281,smooth cycle,variable=\a]
			({3.00+\circlescale*((1-.08*cos(3*\a))*cos(\a)-\packetoffset)},
			 {\cc+\circlescale*(1-.08*cos(3*\a))*sin(\a)});
		\draw[descent qthree]
			plot[domain=0:360,samples=281,smooth cycle,variable=\a]
			({3.00+\circlescale*(.9775*cos(\a)*cos(23.4913)
				-.805*sin(\a)*sin(23.4913)-\packetoffset)},
			 {\cc+\circlescale*(.9775*cos(\a)*sin(23.4913)
				+.805*sin(\a)*cos(23.4913))});

		\foreach \a in {30,90,150,210,270,330}{
			\node[descent node] at
				({3.00+\circlescale*(cos(\a)-\packetoffset)},
				 {\cc+\circlescale*sin(\a)}) {};
		}
		\foreach \xx/\eta in \mixedpacketnodes{
			\node[descent node] at
				({3.00+\circlescale*(\xx-\packetoffset)},
				 {\cc+\circlescale*\eta}) {};
		}

		\foreach \eta in {.929145939,-.929145939}{
			\node[descent orangehit] at
				(3.00,{\cc+\circlescale*\eta}) {};
		}
		\foreach \eta in {1.024366505,-1.024366505}{
			\node[descent bluehit] at
				(3.00,{\cc+\circlescale*\eta}) {};
		}
		\foreach \eta in {.832440513,-.781608539}{
			\node[descent greenhit] at
				(3.00,{\cc+\circlescale*\eta}) {};
		}
		\node[descent limiting] at (3.00,\cc) {};

		\node[font=\scriptsize,anchor=west] at (3.82,\pR) {\(p_R\)};
		\node[font=\scriptsize,anchor=west] at (3.82,\cc) {\(c\)};
	\end{tikzpicture}%
}

	\begingroup
	\edef\savedparindent{\the\parindent}
	\Needspace{11\baselineskip}
	\begin{figure}[ht!]
		\centering
		\begin{minipage}[t]{.38\linewidth}
			\makebox[\linewidth][r]{\small\textup{(a)}}
			\par\vspace{.4mm}
			\tikzsetnextfilename{seven-branch-descent-a}
			\resizebox{\linewidth}{!}{\SevenBranchDescentA}
		\end{minipage}
		\hfill
		\begin{minipage}[t]{.38\linewidth}
			\makebox[\linewidth][r]{\small\textup{(b)}}
			\par\vspace{.4mm}
			\tikzsetnextfilename{seven-branch-descent-b}
			\resizebox{\linewidth}{!}{\SevenBranchDescentB}
		\end{minipage}
	\end{figure}

	\Needspace{11\baselineskip}
	\noindent
	\begin{minipage}[t]{.74\linewidth}
		\vspace{0pt}
		\setlength{\parindent}{\savedparindent}
			After contracting $E_R$, the image of $J_R$ has intersection
			multiplicity $\mt(R,0)=2$ with $E_1$.  A small translation separates
			this contact into the two transverse real intersections shown in
			figure~(c).
		Contracting $E_P$ leaves this chart unchanged.

		Figure (d) shows $J_R$ and the  point $c$ after $E_R$ and
		$E_P$ have been contracted.  The arc $J_R$ is away from $c$ and the
			other six strict transforms pass through $c$.  They are smooth 
			there, so
			$m_{\mathfrak q_i}=1$ for $i = 1, 2, 3$.  Apply
			\Cref{thm:distinct-tangent-pair-families} to
			$\mathcal P_c=\mathcal P_1\sqcup\mathcal P_2$, with $M=1$.
	\end{minipage}
	\hfill
	\begin{minipage}[t]{.21\linewidth}
		\vspace{0pt}
		\makebox[\linewidth][r]{\small\textup{(c)}}
		\par\vspace{.4mm}
		\tikzsetnextfilename{seven-branch-descent-c}
		\resizebox{\linewidth}{!}{\SevenBranchDescentC}
	\end{minipage}
	\par\smallskip

	\Needspace{13\baselineskip}
	\noindent
	\begin{minipage}[t]{.15\linewidth}
		\vspace{0pt}
		\makebox[\linewidth][r]{\small\textup{(d)}}
		\par\vspace{.4mm}
		\tikzsetnextfilename{seven-branch-descent-d}
		\resizebox{\linewidth}{!}{\SevenBranchDescentD}
	\end{minipage}
	\hfill
	\begin{minipage}[t]{.79\linewidth}
		\vspace{0pt}
		\setlength{\parindent}{\savedparindent}
			Choose $\kappa_1, \kappa_2 > 0$ as in
			\Cref{thm:distinct-tangent-pair-families}. And let $N$ be
			large enough compared with the power used in the previous 
			translation.
			For $\lambda>0$ small enough, set
			\[
				r_i(\lambda)=\kappa_j\lambda^N
				\quad\text{if }\mathfrak q_i\in\mathcal P_j.
			\]
			Let $\Gamma\subset E_{1,\R}$ be a small interval around $c$,
			regarded as one of the real arcs used in
			\Cref{thm:distinct-tangent-pair-families}.  For each $i$, let
			\[
				A_{i,\lambda}
				=
				\operatorname{Im}\!
				\left(
					\alpha^{\mathrm{cp}}_{\boldsymbol\gamma_{\mathfrak q_i},
					r_i(\lambda)}
				\right)
			\]
			be the image of the circle map associated with
			$\mathfrak q_i$.  With
			$\tau=\lambda^N$, that same lemma gives
		\[
		\begin{aligned}
			\dbl(A_{i,\lambda})&=0,
			\qquad i=1,2,3,\\
			\dbl(A_{1,\lambda},A_{2,\lambda})&=2+2+1+1=6,\\
			\dbl(A_{1,\lambda},A_{3,\lambda})
			=\dbl(A_{2,\lambda},A_{3,\lambda})&=4,\\
			\dbl(A_{i,\lambda},\Gamma)&=2,
			\qquad i=1,2,3.
		\end{aligned}
		\]
	\end{minipage}
	\par\smallskip

	So the images of the three circles have
	$6+4+4=14$ intersection points, and they meet $\Gamma$ in six points.

	\Needspace{14\baselineskip}
	\noindent
	\begin{minipage}[t]{.74\linewidth}
		\vspace{0pt}
		\setlength{\parindent}{\savedparindent}
			Next we apply the translation from
			step~\ref{it:global-algorithm-translations} using a higher power of
			$\lambda$.  This translation moves the multiple points away from 
			$E_1$ and ensures $6$ intersection points of the circles with the 
			divisor.

		Figure (e) shows the situation before $E_1$ is contracted.  Orange,
		blue, and green represent
		$A_{1, \lambda}, A_{2, \lambda}$ and $A_{3, \lambda}$ respectively.  
		The six intersection points of these three divides with
		with $\Gamma$, together with the two points of $J_R\cap E_1$, give us
			eight distinct real points on $E_1$.  The other $14$ intersection
			are not in $E_1$.

		Let $O$ be the image of $E_1$ by its blowdown.  The eight curve
		germs meeting $E_1$ at these points descend to smooth real germs through
		$O$ with distinct tangent lines.  So they form an ordinary real point of
		multiplicity eight (see
		\Cref{fig:seven-branch-final-contraction}).
	\end{minipage}
	\hfill
	\begin{minipage}[t]{.21\linewidth}
		\vspace{0pt}
		\makebox[\linewidth][r]{\small\textup{(e)}}
		\par\vspace{.4mm}
		\tikzsetnextfilename{seven-branch-descent-e}
		\resizebox{\linewidth}{!}{\SevenBranchDescentE}
	\end{minipage}
	\par\smallskip
	\endgroup

	\begin{figure}[!ht]
		\centering
		\tikzsetnextfilename{seven-branch-final-predivide}
		\resizebox{0.72\linewidth}{!}{%
			\input{figures/seven_branch_final_predivide.tikz}%
		}
		\caption{Contraction of $E_1$.  The eight intersections with $E_1$
		map to $O$ and the other $14$ intersection points remain away from
		$O$.}
		\label{fig:seven-branch-final-contraction}
	\end{figure}

	The final perturbation (as in step \ref{it:global-algorithm-count}) 
	produces fourteen ordinary double points away from $O$ and $\binom82=28$ 
	nodes near $O$.  Thus there are
	\[
		14+\binom82=42=\delta(C,0)-\imbr(C,0)
	\]
	nodes.  See
	\Cref{fig:seven-branch-descent-splitting} for the perturbation near $0$.

	\begin{figure}[!ht]
		\centering
		\tikzsetnextfilename{seven-branch-descent-and-splitting}
		\resizebox{0.6\linewidth}{!}{%
			\begin{tikzpicture}[
	x=1cm,
	y=1cm,
	line cap=round,
	line join=round,
	>=stealth
]
\definecolor{orbitone}{HTML}{D55E00}
\definecolor{orbittwo}{HTML}{0072B2}
\definecolor{orbitthree}{HTML}{009E73}
\definecolor{realorbit}{HTML}{7A3E9D}
\definecolor{nodecharcoal}{HTML}{252525}
\path[use as bounding box] (6.10,-1.82) -- (15.15,2.22);
\tikzset{
	ordinary/.style={circle,fill=nodecharcoal,inner sep=1.25pt}
}

\begin{scope}[shift={(8.10,.20)}]
	\begin{scope}[scale=1.24]
	\foreach \m/\col in {
		-1.15/realorbit,-.85/realorbit,
		.45/orbitone,.65/orbittwo,.85/orbitthree,
		1.15/orbitthree,1.35/orbittwo,1.55/orbitone}{
		\pgfmathsetmacro{\dx}{1.46/sqrt(1+\m*\m)}
		\draw[\col,line width=.85pt]
			({-\dx},{-\m*\dx}) -- (\dx,{\m*\dx});
	}
	\node[ordinary] at (0,0) {};
	\end{scope}
\end{scope}

\draw[->,gray!72,line width=.65pt] (9.98,.20) -- (10.64,.20);

\begin{scope}[shift={(13.10,.20)}]
	\begin{scope}[scale=1.24]
	\clip (0,0) circle (1.53);
	\foreach \m/\b/\col in {
		-1.15/1.037608/realorbit,
		-.85/.743649/realorbit,
		.45/-.007868/orbitone,
		.65/-.043425/orbittwo,
		.85/-.045394/orbitthree,
		1.15/.075701/orbitthree,
		1.35/.260246/orbittwo,
		1.55/.421274/orbitone}{
		\draw[\col,line width=.85pt]
			(-1.75,{-1.75*\m+\b}) -- (1.75,{1.75*\m+\b});
	}
	\foreach \mi/\bi [count=\ii] in {
		-1.15/1.037608,-.85/.743649,.45/-.007868,.65/-.043425,
		.85/-.045394,1.15/.075701,1.35/.260246,1.55/.421274}{
		\foreach \mj/\bj [count=\jj] in {
			-1.15/1.037608,-.85/.743649,.45/-.007868,.65/-.043425,
			.85/-.045394,1.15/.075701,1.35/.260246,1.55/.421274}{
			\ifnum\ii<\jj
				\pgfmathsetmacro{\xx}{(\bj-\bi)/(\mi-\mj)}
				\pgfmathsetmacro{\yy}{\mi*\xx+\bi}
				\fill[white] (\xx,\yy) circle (1.00pt);
				\fill[nodecharcoal] (\xx,\yy) circle (.55pt);
			\fi
		}
	}
	\end{scope}
\end{scope}
\end{tikzpicture}%
		}
		\caption{The perturbation of the ordinary point of multiplicity eight 
		at $O$
		produces $\binom82=28$ nodes.}
		\label{fig:seven-branch-descent-splitting}
	\end{figure}
	\FloatBarrier

	After making the finite base changes in
	\Cref{eq:global-base-change}, and again writing $\lambda$ for the final
	parameter, the analytic realization in the proof of
	\Cref{thm:global-predivide} gives the existence of an equation
	\[
		F_{\mathrm{an}}\in\R\{x,y,\lambda\},
		\qquad
		F_{\mathrm{an}}(x,y,0)=f(x,y)
	\]
	defining the real deformation that we have just constructed.
	Here we explain how to obtain this equation in practice.  Put $s=\lambda^2$.
	For the real branch, set $T=u+v$ and $
		\xi=T^2-3s$ with 
		$\zeta=T\xi$. 
	For each conjugate pair we have $uv=s$ and we set
	\[
		\zeta=u+v,
		\qquad
		D=u-v,
		\qquad
		D^2=\zeta^2-4s.
	\]
	For $\mathfrak q_2$, translate the coordinate $z$ by $\lambda^3$.
	After the blowdown $(y,z)\mapsto(yz,y)$, the first coordinates for
	$\mathfrak q_1,\mathfrak q_2,\mathfrak q_3$ are respectively
	\[
	\begin{aligned}
		\xi&=\zeta(\zeta^2-2s)+i\zeta D,\\
		\xi&=\zeta(2\zeta^2-4s+\lambda^3)+i\zeta D,\\
		\xi&=2\zeta^2+3i\zeta D.
	\end{aligned}
	\]
	Eliminating $T$ and $D$ from these formulas gives the following polynomials
	\[
	\begin{aligned}
		h_{0,\lambda}(\xi,\zeta)
		&=\zeta^2-\xi^2(\xi+3s),\\
		h_{1,\lambda}(\xi,\zeta)
		&=\bigl[\xi-\zeta(\zeta^2-2s)\bigr]^2
		  +\zeta^2(\zeta^2-4s),\\
		h_{2,\lambda}(\xi,\zeta)
		&=\bigl[\xi-\zeta(2\zeta^2-4s+\lambda^3)\bigr]^2
		  +\zeta^2(\zeta^2-4s),\\
		h_{3,\lambda}(\xi,\zeta)
		&=(\xi-2\zeta^2)^2+9\zeta^2(\zeta^2-4s).
	\end{aligned}
	\]
	Set $\eta=\lambda^{20}$, and define
	\[
		Z_{k,\lambda}
		=
		\bigl\{(x,y):
		h_{k,\lambda}(x-k\eta,y-k^2\eta)=0\bigr\},
		\quad k=0,1,2,3.
	\]
	The polynomial
	\begin{equation}
		\label{eq:seven-branch-explicit-morsification}
		F_{\mathrm{poly}}(x,y,\lambda)
		=
		\prod_{k=0}^3
		h_{k,\lambda}(x-k\eta,y-k^2\eta)
	\end{equation}
	defines the zero set of the four polynomials above.
	Setting $\lambda=0$ gives
	$F_{\mathrm{poly}}(x,y,0)=f(x,y)$.
	Hence $\lambda$ is a nonzero divisor in
	$\C\{x, y, \lambda\}/(F_{\mathrm{poly}})$. And so
	$\{F_{\mathrm{poly}}=0\}$ is flat over $\Delta_\lambda$.  

	\bibliographystyle{alpha}
	\bibliography{biblio/references}
\end{document}